\documentclass{amsart}

\usepackage[utf8]{inputenc}
\usepackage[english]{babel}
\usepackage{algorithm}
\usepackage{algpseudocode}
\usepackage{amsmath}
\usepackage{amssymb}
\usepackage{amsthm}
\usepackage{booktabs}
\usepackage{csquotes}
\usepackage{geometry}
\usepackage{graphicx}
\usepackage{mathtools}
\usepackage{svg}
\usepackage{tikz}
\usetikzlibrary{calc,shapes.geometric,arrows.meta}
\usepackage{todonotes}

\usepackage{color}
\usepackage{cleveref}
\usepackage{graphics}
\title[Representations of Gaussian and Eisenstein Integers]{Optimal Representations of Gaussian  and Eisenstein Integers using digit sets closed under multiplication}

\author[1]{Adam Blažek, Edita Pelantová,   and Milena Svobodová }
\def\N{\mathbb N}
\def\Z{\mathbb Z}

\def\C{\mathbb C}

\def\B{\mathcal B}
\def\CC{\mathcal C}
\def\D{\mathcal D}
\def\P{\mathcal P}

\def\ii{\imath}

\def\G{\widetilde{G}}
\def\GG{\widetilde{\Gamma}}

\newtheorem{thm}{Theorem}[section]
\newtheorem{theorem}[thm]{Theorem}

\newtheorem{corollary}[thm]{Corollary}
\newtheorem{lemma}[thm]{Lemma}

\newtheorem{proposition}[thm]{Proposition}

\newtheorem{definition}[thm]{Definition}

\crefname{thm}{theorem}{theorems}
\crefname{theorem}{theorem}{theorems}
\crefname{coro}{corollary}{corollaries}
\crefname{example}{example}{examples}
\crefname{lemma}{lemma}{lemmas}
\crefname{lmm}{lemma}{lemmas}
\crefname{claim}{claim}{claims}
\crefname{obs}{observation}{observations}
\crefname{proposition}{proposition}{propositions}
\crefname{prop}{proposition}{propositions}
\crefname{definition}{definition}{definitions}

\theoremstyle{remark}
\newtheorem{remark}[thm]{Remark}

\theoremstyle{example}
\newtheorem{example}[thm]{Example}
\crefname{example}{example}{examples}

\definecolor{zero}{HTML}{666666}
\definecolor{0}{HTML}{0C9E00}
\definecolor{60}{HTML}{0066FF}
\definecolor{90}{HTML}{0066FF}
\definecolor{120}{HTML}{8900FF}
\definecolor{180}{HTML}{FF0000}
\definecolor{240}{HTML}{FF7F00}
\definecolor{270}{HTML}{DEC100}
\definecolor{300}{HTML}{DEC100}
\tikzstyle{zero}=[-{Latex[length=5pt]},thick,solid,color=zero]
\tikzstyle{0}=[-{Latex[length=5pt]},thick,densely dashed,color=0]
\tikzstyle{60}=[-{Latex[length=5pt]},thick,dash dot,color=60]
\tikzstyle{90}=[-{Latex[length=5pt]},thick,dash dot,color=90]
\tikzstyle{120}=[-{Latex[length=5pt]},very thin,densely dashed,color=120]
\tikzstyle{180}=[-{Latex[length=5pt]},thick,densely dotted,color=180]
\tikzstyle{240}=[-{Latex[length=5pt]},very thin,solid,color=240]
\tikzstyle{270}=[-{Latex[length=5pt]},thick,densely dash dot dot,color=270]
\tikzstyle{300}=[-{Latex[length=5pt]},thick,densely dash dot dot,color=300]

\begin{document}

\maketitle

%%%%%%%%%%%%%%%%%%%%%%%%%%%%%%%%%%%%%%%%%%%%%%%%%%%%%%%%%%%%%%%%%%%%%%%%%%%%%%%%%%%%%%%%%%%%%%%%%%%%
%%%%%%%%%%%%%%%%%%%%%%%%%%%%%%%%%%%%%%%%%%%%%%%%%%%%%%%%%%%%%%%%%%%%%%%%%%%%%%%%%%%%%%%%%%%%%%%%%%%%

\begin{abstract}

We study two positional numeration systems which are known for allowing very efficient addition and multiplication of complex numbers. The first one uses the base $\beta = \ii - 1$ and the digit set $\D = \{ 0, \pm 1, \pm \ii \}$. In this numeration system, every non-zero Gaussian integer~$x$ has an infinite number of representations. We focus on  optimal representations of~$x$ -- i.e., representations with minimal possible number of non-zero digits. One of the optimal representations of~$x$ has the so-called $3$-non-adjacent form ($3$-NAF). We provide an upper bound on the number of distinct optimal representations of~$x$, depending on the number of non-zero digits in the $3$-NAF of~$x$. We also characterize the Gaussian integers for which the upper bound is attained. 
The same questions are answered also for the second numeration system  with base $\beta = \omega - 1$ and digit set $\D = \{ 0, \pm 1, \pm \omega, \pm \omega^2 \}$, where $\omega = \exp(2\pi\ii / 3)$. In this system, every Eisenstein integer has a $2$-NAF, which is optimal. 
This paper can be understood as an analogy to the result of Grabner and Heuberger obtained for the signed binary numeration system, using base $\beta = 2$ and digit set $\D = \{0, \pm 1\}$.
\end{abstract}

\bigskip

\centerline{Affiliation: Department of Mathematics,Czech Technical University in Prague }

\bigskip

%%%%%%%%%%%%%%%%%%%%%%%%%%%%%%%%%%%%%%%%%%%%%%%%%%%%%%%%%%%%%%%%%%%%%%%%%%%%%%%%%%%%%%%%%%%%%%%%%%%%
%%%%%%%%%%%%%%%%%%%%%%%%%%%%%%%%%%%%%%%%%%%%%%%%%%%%%%%%%%%%%%%%%%%%%%%%%%%%%%%%%%%%%%%%%%%%%%%%%%%%
\section{Introduction}
%%%%%%%%%%%%%%%%%%%%%%%%%%%%%%%%%%%%%%%%%%%%%%%%%%%%%%%%%%%%%%%%%%%%%%%%%%%%%%%%%%%%%%%%%%%%%%%%%%%%

Methods for speeding up arithmetical operations have been in the center of research for mathematicians and engineers since the 1950's. It is in cryptography, but also in other tasks requiring massive calculations, that these methods are finding their applications.

G.~W.~Reitwiesner \cite{R1960} has shown that any integer~$x$ can be represented in the numeration system with base $\beta = 2$ and digit set $\D = \{-1, 0, 1\}$ in the form $x = \sum_{k=0}^{N} d_k 2^k$ in such a way that at most one out of two neighbouring digits $d_j$ and $d_{j-1}$ is non-zero for each $j$. Such a representation is unique for a given integer~$x$, and, moreover, it has the minimum Hamming weight amongst all representations of~$x$ in this numeration system $(\beta, \D)$. Contrary to the classic binary numeration system with digit set $\{0, 1\}$, where half of the digits in number representations are non-zero on average, the representations introduced by Reitwiesner have just one third   of non-zero digits on average. This decreases the number of additions needed in processing of multiplication.

The idea of Reitwiesner was generalized into the notion of so-called {\it width-$W$ non-adjacent form} -- i.e., representing numbers in a basis and digit set chosen conveniently so that any~$W$ consecutive  digits contain at most one non-zero digit. Such representations are called $W$-NAF representations. For imaginary quadratic rings, which are in the center of our interest, a detailed analysis of relations between the bases, digit sets and width~$W$ was performed by C.~Heuberger and D.~Krenn in \cite{HK2013} and~\cite{HK2013b}.

The on-line algorithm for multiplication of real numbers proposed by K.~S.~Trivedi and M.~D.~Ercegovac~\cite{TrEr} obtains the $n$~most significant digits of the result in linear time with respect to~$n$. Both input and output numbers are represented in a~numeration system with base $\beta \in \C$, $|\beta| > 1$ and digit set~$\D$ allowing addition in parallel, as firstly introduced by A.~Avizienis~\cite{Avi}. Addition is performed in a constant number of steps, irrespective of the length of the summands. That is possible, however, only when the digit set is rich enough. Suitable alphabets for parallel addition in integer bases were described by C.~Y.~Chow and J.~E.~Robertson \cite{ChR1978} and by B.~Parhami~\cite{Parhami}, and minimal alphabets for quadratic bases can be found in~\cite{LS2019}.

\medskip

Our goal is to draw attention to two numeration systems for representing complex numbers. These systems are defined by their bases and digit sets as follows:
\begin{eqnarray}
    \label{D:Penney-integers}
    \beta = \ii - 1 & \text{ and } & \D = \{ 0, \pm 1, \pm \ii \}, \\
    \label{D:Eisenstein-integers}
    \beta = \omega - 1 & \text{ and } & \D = \{ 0, \pm 1, \pm \omega, \pm \omega^2 \}, \text{ where } \omega = \exp(2\pi\ii / 3).
\end{eqnarray}
Both these numeration systems can express any complex number~$x$ in the form $x = \sum_{k  =-\infty}^N d_k \beta^k$, where $d_k \in \D$.

When focusing on numbers with representations containing only non-negative powers of the base~$\beta$, we get
$$
\Bigl\{\sum \limits_{k=0}^N d_k \beta^k \, : \, N \in \N \text{ and } d_k \in \D \Bigr\} = \left\{ 
\begin{array}{cl}
    \mathbb{Z}[\ii]     & \text{= the ring of Gaussian integers in case (1),}\\
    \mathbb{Z}[\omega]  & \text{= the ring of Eisenstein integers in case (2).}
\end{array} \right.
$$

These two systems are exceptional when it comes to efficiency of multiplication algorithm. The efficiency stems from the following three properties of $(\beta, \D)$:
\begin{itemize}
    \item the digit set $\D \subset \C$, in both cases (1) and~(2), is closed under multiplication and thus the operation of multiplication by a single digit is carry-free;   
    \item the $3$-NAF representation of any element~$x$ of the ring $\Z[\ii]$ is available, and at the same time it is minimal with respect to the number of non-zero digits in the representation; the same is true for the $2$-NAF representation of any element of the ring $\Z[\omega]$;
    \item addition in $(\beta, \D)$ is doable in parallel, and also efficient on-line multiplication is possible. 
\end{itemize} 
The $W$-NAF representation of~$x$ has the minimal Hamming weight, but the same Hamming weight can be achieved also with other representations of ~$x$ in the system $(\beta, \D)$. Representations with the minimal Hamming weight shall be called {\it optimal}. Using a randomly chosen optimal representation of~$x$, instead of its $W$-NAF representation, can improve resistance of the elliptic curve crypto-systems against differential power analysis. For the numeration system with base $\beta = 2$ and digit set $\D = \{0, \pm 1\}$, a sharp upper bound on the number of optimal representations of an integer $n$ was deduced by Grabner and Heuberger in \cite{GraHeu2006}. In 2010, Wu et al.~\cite{WZWD2010} provided an upper bound on the number of optimal representations, their bound depends on length of the $2$-NAF representation of~$x$.  

\medskip

In our contribution, we first compare the two systems mentioned in cases \eqref{D:Penney-integers} and~\eqref{D:Eisenstein-integers} from the perspective of the density of non-zero digits in representations of complex numbers. We show that the first one is slightly more advantageous in this respect.

For each of the two numeration systems $(\beta, \D)$, we assign a matrix~$A_d$ to every digit $d \in \D$, and then derive a matrix formula for enumeration of $f(x)$ = the number of optimal representations of $x \in \mathbb{Z}[\beta]$. In fact, we show that $f(x)$ is a recognisable series as introduced in \cite{BeRe2010} and the family of matrices $A_d$ in $\Z^{9 \times 9}$ and $\Z^{7 \times 7}$, for Gaussian and Eisenstein system respectively, together with a row vector and a column vector, form their linear representations.  

We provide a sharp upper bound on $f(x)$ depending on the number of non-zero digits in the $W$-NAF representation of~$x$. The task of finding this bound is transformed into a related task to determine the joint spectral radius for the set of matrices $\{A_d : d \in \D\}$. We propose a graph in which a random walk assigns an optimal representation of~$x$, starting from the $W$-NAF representation of~$x$. 
 
In Section \ref{sec:Penney}, we study in detail all optimal representations in the numeration system~$(1)$ for Gaussian integers. The digit set~$\D$ used in the numeration system~$(2)$ for representing Eisenstein integers has more symmetries than the digit set for system~$(1)$. That greatly simplifies the process of deriving the properties of system~$(2)$, therefore, in Section~\ref{sec:Eisenstein} we provide just a sketch of the proofs.

%%%%%%%%%%%%%%%%%%%%%%%%%%%%%%%%%%%%%%%%%%%%%%%%%%%%%%%%%%%%%%%%%%%%%%%%%%%%%%%%%%%%%%%%%%%%%%%%%%%%
%%%%%%%%%%%%%%%%%%%%%%%%%%%%%%%%%%%%%%%%%%%%%%%%%%%%%%%%%%%%%%%%%%%%%%%%%%%%%%%%%%%%%%%%%%%%%%%%%%%%
\section{Known results on two selected number systems} 
%%%%%%%%%%%%%%%%%%%%%%%%%%%%%%%%%%%%%%%%%%%%%%%%%%%%%%%%%%%%%%%%%%%%%%%%%%%%%%%%%%%%%%%%%%%%%%%%%%%%

The two bases $\beta \in \C$ that we study are imaginary quadratic. In particular,
\begin{itemize}
    \item the first base $\beta = \ii - 1$ is a root of the minimal polynomial $X^2 + 2X + 2$ and $\Z[\beta] = \Z[\ii]$ is the ring of {\it Gaussian integers};
    \item the second base $\beta = \omega - 1$ is a root of $X^2 + 3X + 3$ and $\Z[\beta] = \Z[\omega]$ is the ring of {\it Eisenstein integers}. 
\end{itemize}

In accordance with terminology of C.~Heuberger and D.~Krenn, the pair consisting of the base~$\beta$ and of the digit set~$\D$ is called {\it numeration system}, when each element of the ring $\Z[\beta]$ can be expressed in the form $\sum_{k=0}^N d_k \beta_k$. It is clear that if $(\beta, \D)$ forms a numeration system, then the number of elements in the digit set~$\D$ must be at least equal to  the norm of~$\beta$ -- i.e., $N(\beta) = \beta \overline{\beta}$. Moreover, the digit set must contain at least one representative from each congruence class modulo~$\beta$. Still, fulfilment of both these conditions is not sufficient to form a numeration system.

W.~Penney in~\cite{P1965} has shown that the base $\beta = \ii - 1$ with the binary digit set $\{0, 1\}$ constitute a~numeration system, wherein, of course, the set $\{0, 1\}$ is closed under multiplication. Note that, on the contrary, a similar base $\beta = \ii + 1$ with the same digit set does not form a numeration system.

The base $\beta = \omega - 1$ has norm $N(\beta) = 3$ and forms a~numeration system together with digit sets $\{0, 1, 2\}$ or $\{0, 1, -\omega\}$, but these digit sets are not closed under multiplication. On the other hand, $\beta = \omega - 1$ with the digit set $\{0, \pm 1\}$, which is closed under multiplication, do not form a numeration system; for instance, $x = \overline{\omega}$ cannot be expressed as $\sum_{k=0}^N d_k \beta^k$.

When $(\beta, \D)$ is a numeration system and $\#\D$ is equal to the norm~$N(\beta)$, then each element of the ring $\Z[\beta]$ has a unique $(\beta, \D)$-representation. This property is not necessarily advantageous. Redundant digit sets can enable representing numbers using a low number of non-zero digits, and also defining better algorithms for processing arithmetic operations. The following Theorem from \cite{FrPeSv} and~\cite{L2018} shows the necessary level of redundancy for parallel addition:

%%%%%%%%%%%%%%%%%%%%%%%%%%%%%%%%%%%%%%%%%%%%%%%%%%
\begin{theorem}\label{thm:paralel}
Let $\beta \in \C$, $|\beta| > 1$ be an algebraic integer with the minimal polynomial $f(X)$, and let $\D \subset \Z[\beta]$. If the system $(\beta, \D)$ allows parallel addition, then $\# \D \geq |f(1)|$. 
\end{theorem}
%%%%%%%%%%%%%%%%%%%%%%%%%%%%%%%%%%%%%%%%%%%%%%%%%%

The digit set $\{0, \pm 1\}$ is closed under multiplication, and it is redundant with respect to the base $\beta = \ii - 1$. The minimal polynomial of $\beta$ is $f(X) = X^2 + 2X + 2$, and thus $f(1) = 5$. According to Theorem \ref{thm:paralel}, the digit set $\{0, \pm 1\}$ does not enable parallel addition.

Similarly, the digit set $\{0, 1, \omega, \omega^2\}$ with $\omega = \exp(2\pi\ii / 3)$ is closed under multiplication, and it is redundant with respect to the base $\beta = \omega - 1$. But, again, this digit set does not enable parallel addition for this base $\beta$, because the minimal polynomial of $\beta$ is $f(X) = X^2 + 3X + 3$, and thus $f(1) = 7$.

%%%%%%%%%%%%%%%%%%%%%%%%%%%%%%%%%%%%%%%%%%%%%%%%%%
\begin{definition}\label{de:minimalNorm}
Let $\beta$ be an algebraic integer, imaginary quadratic, $W \in \N$, $W \geq 2$. The set $\D \subset \Z[\beta]$ is called a {\it minimal norm representative digit set with respect to~$W$}, if it consists of $0$ and exactly one representative from each residue class of $\Z[\beta]$ modulo~$\beta^W$ that is not divisible by~$\beta$, and, moreover, each representative contained in~$\D$ is of minimal norm in its residue class.
\end{definition}
%%%%%%%%%%%%%%%%%%%%%%%%%%%%%%%%%%%%%%%%%%%%%%%%%%

%%%%%%%%%%%%%%%%%%%%%%%%%%%%%%%%%%%%%%%%%%%%%%%%%%
\begin{example}\label{E:MinRepDigSets_Penney+Eisenstein}
Let us find the minimal norm representative digit sets for our two favourite bases:
\begin{description}
    \item[$\beta = \ii - 1$ and $W =3$] \qquad As $\beta^3 = 2 + 2 \ii$, there are $N(\beta^3) = (\beta \overline{\beta})^3 = 8$ residual classes modulo~$\beta^3$ with representatives of minimal norm: $0, \pm 1, \pm \ii, 2, \ii \pm 1$. Three of these representatives $2 =  \beta \overline{\beta}$, $\ii + 1 = -\ii \beta$ and $\ii - 1 = \beta$ are non-zero and divisible by $\beta$. Hence
        $$\D = \{0,\pm 1, \pm \ii\} \, .$$ 
    \item[$\beta = \omega - 1$ and $W =2$] \qquad As $\beta^2 = -3 \omega$, there are $N(\beta^2) = (\beta \overline{\beta})^2 = 9$ residual classes modulo~$\beta^2$. Two representatives of minimal norm, namely $\pm(\omega - 1)$, are non-zero and divisible by~$\beta$. The remaining seven representatives form the digit set
        $$\D = \{0,\pm 1, \pm \omega, \pm \omega^2\} \, .$$
\end{description}
\end{example}
%%%%%%%%%%%%%%%%%%%%%%%%%%%%%%%%%%%%%%%%%%%%%%%%%%

The minimal norm representative digit sets $\D$ presented in Example~\ref{E:MinRepDigSets_Penney+Eisenstein} are closed under multiplication, and they also have the cardinality required by Theorem~\ref{thm:paralel} for parallel addition. Another exceptional property of these digit sets follows from the result of C.~Heuberger and D.~Krenn, see Theorem~9.1 and Corollary~8.1 in~\cite{HK2013}.

%%%%%%%%%%%%%%%%%%%%%%%%%%%%%%%%%%%%%%%%%%%%%%%%%%
\begin{theorem}\label{thm:MinimalHamming}
Let $(\beta, \D)$  be a numeration system  and $s \in \mathbb{N}$, where 
\begin{enumerate}
    \item either $\beta = \ii - 1$, $\D = \{0, \pm 1, \pm \ii\}$ and $W = 3$, 
    \item  or $\beta = \omega - 1$, $\D = \{0, \pm 1, \pm \omega, \pm \omega^2\}$ and $W = 2$.
\end{enumerate} 
Then every $x \in \mathbb{Z}[\beta]$ has $W$-NAF representation in $(\beta, \D)$ and this representation has the minimal Hamming weight among all representations of~$x$ in $(\beta, \D)$.     
\end{theorem}
%%%%%%%%%%%%%%%%%%%%%%%%%%%%%%%%%%%%%%%%%%%%%%%%%%

Third exceptional property of the two numeration systems follows from the results \cite{LS2019} and~\cite{FrPaPeSv}:

%%%%%%%%%%%%%%%%%%%%%%%%%%%%%%%%%%%%%%%%%%%%%%%%%%
\begin{theorem}
Let $(\beta, \D)$  be a numeration system, where 
\begin{enumerate}
    \item either $\beta = \ii - 1$ and $\D = \{0, \pm 1, \pm \ii \}$,
    \item or $\beta = \omega - 1$ and $\D = \{0,\pm 1, \pm \omega, \pm \omega^2\}$. 
\end{enumerate}
Then the numeration system $(\beta, \D)$ allows parallel addition and on-line multiplication. 
\end{theorem}
%%%%%%%%%%%%%%%%%%%%%%%%%%%%%%%%%%%%%%%%%%%%%%%%%%

When $x \in \Z[\beta]$ can be expressed in the numeration system $(\beta, \D)$ as $x = \sum_{k=0}^N d_k \beta^k$ with $d_k \in \D$, we denote it by $(x)_{\beta, \D} = d_N d_{N-1} \cdots d_0$. If the numeration system is clear from the context, we identify~$x$ with the string, and write simply $x = d_N d_{N-1} \cdots d_0$.

However, we are interested especially in such representations that contain the minimal number of non-zero digits. The frequency of non-zero digits appearing in $W$-NAF representations for our two numeration systems can be found in~\cite{HK2013b}. Let us first specify the term {\it frequency} for a given numeration system $(\beta, \D)$, allowing $W$-NAF representation for each element $x \in \Z[\beta]$.

Let us consider $N \in \N$, and denote by $S(N)$ the number of all $x \in \Z[\beta]$ whose $W$-NAF representation is a string of length~$\leq N$, and denote by $R(N)$ the number of non-zero digits in $W$-NAF representations of all such numbers~$x$. Then the limit $\lim\limits_{N \to \infty} \frac{R(N)}{N S(N)}$, if it exists, is called the {\it frequency} of non-zero digits in the $W$-NAF representation, and we denote it by~$f_{\neq 0}(\beta, \mathcal{D})$.   

%%%%%%%%%%%%%%%%%%%%%%%%%%%%%%%%%%%%%%%%%%%%%%%%%%
\begin{proposition} \qquad
\begin{enumerate}
    \item If $\beta = \ii - 1$ and $\D = \{0, \pm 1, \pm \ii\}$, then the frequency $f_{\neq 0}(\beta, \D)$ of non-zero digits of $3$-NAF representations equals~$\frac14$.  
    \item If $\beta = \omega - 1$ and $\D = \{0, \pm 1, \pm \omega, \pm \omega^2\}$, then the frequency $f_{\neq 0}(\beta, \D)$ of non-zero digits of $2$-NAF representations equals~$\frac25$.   
\end{enumerate}    
\end{proposition}
%%%%%%%%%%%%%%%%%%%%%%%%%%%%%%%%%%%%%%%%%%%%%%%%%%

It seems clear that the base $\beta = \ii - 1$ should be more advantageous for representing complex numbers. We have to consider, however, that the length of the $W$-NAF string representing $x \in \Z[\beta]$ is approximately $\log_{|\beta|} |x|$. Hence, the $W$-NAF representation of~$x$ has $f_{\neq 0} (\beta, \D) \times \log_{|\beta|} |x|$ non-zero digits on average.
\begin{itemize}
    \item For system $\beta = \ii - 1$ with $\D = \{0, \pm 1, \pm \ii\}$, we thus have the value $V_{{\ii - 1}} = \frac14 \frac{2} {\ln 2} \ln |x|$.
    \item For system $\beta = \omega - 1$ with $\D = \{0,\pm 1, \pm \omega, \pm \omega^2\}$, we have the value $V_{\omega - 1} = \frac25 \frac{2}{\ln 3} \ln |x|$. 
\end{itemize}
As the ratio $V_{\omega - 1} / V_{\ii -1} \approx 1.0094876$ is bigger than~$1$, the system $\beta = \ii - 1$ with $\D = \{0, \pm 1, \pm \ii\}$ can be considered just slightly more effective with respect to the average number of non-zero digits.

%%%%%%%%%%%%%%%%%%%%%%%%%%%%%%%%%%%%%%%%%%%%%%%%%%%%%%%%%%%%%%%%%%%%%%%%%%%%%%%%%%%%%%%%%%%%%%%%%%%%
%%%%%%%%%%%%%%%%%%%%%%%%%%%%%%%%%%%%%%%%%%%%%%%%%%%%%%%%%%%%%%%%%%%%%%%%%%%%%%%%%%%%%%%%%%%%%%%%%%%%
\section{Representations of Gaussian Integers}\label{sec:Penney}
%%%%%%%%%%%%%%%%%%%%%%%%%%%%%%%%%%%%%%%%%%%%%%%%%%%%%%%%%%%%%%%%%%%%%%%%%%%%%%%%%%%%%%%%%%%%%%%%%%%%

%%%%%%%%%%%%%%%%%%%%%%%%%%%%%%%%%%%%%%%%%%%%%%%%%%%%%%%%%%%%%%%%%%%%%%%%%%%%%%%%%%%%%%%%%%%%%%%%%%%%

%%%%%%%%%%%%%%%%%%%%%%%%%%%%%%%%%%%%%%%%%%%%%%%%%%%%%%%%%%%%%%%%%%%%%%%%%%%%%%%%%%%%%%%%%%%%%%%%%%%%

In this section, we consider the numeration system with base $\beta = \ii-1$ and digit set $\D = \{0, \pm 1, \pm \ii\}$. For easier notation, we adopt the following convention:

\smallskip

\noindent {\bf Convention:}
\begin{itemize}
    \item For complex conjugation of a number~$x$, we have the usual notation~$\overline{x}$.
    \item Besides, in digit strings representing a number~$x$ in the numeration system $(\beta, \D)$, we shall denote the digits $-1$ and $-\ii$ by $\underline{1}$ and $\underline{\ii}$, respectively.
\end{itemize}
For instance, number $x = 1$ is represented by strings $1$, $\underline{\ii} \underline{\ii}$, $\ii 0 \underline{1}$, or $\underline{\ii} 0 1 0 \underline{\ii} 0 \underline{1} 0 \ii 0 1 0 \underline{\ii} 0 \underline{1}$.

\smallskip

In the sequel, we describe a transducer which converts an arbitrary $(\beta, \D)$-representation of a Gaussian integer $x \in \Z[\ii]$ into its $3$-NAF representation in $(\beta, \D)$. The construction of this transducer is well known; we provided it here because it serves as the starting point for our main task -- namely to generate optimal representations of~$x$ from its $3$-NAF representation. Let us recall that a~$(\beta, \D)$-representation of a Gaussian integer $x \in \Z[\ii]$ is optimal if its number of non-zero digits is minimal among all $(\beta, \D)$-representations of~$x$ -- i.e., it is equal to the number of non-zero digits in the $3$-NAF representation of~$x$.

%%%%%%%%%%%%%%%%%%%%%%%%%%%%%%%%%%%%%%%%%%%%%%%%%%%%%%%%%%%%%%%%%%%%%%%%%%%%%%%%%%%%%%%%%%%%%%%%%%%%
\smallskip
\
\subsection{Transducer converting $(\beta,\D)$-representations to 3-NAF}\label{sS:transducer-construction}
\
\medskip
%%%%%%%%%%%%%%%%%%%%%%%%%%%%%%%%%%%%%%%%%%%%%%%%%%%%%%%%%%%%%%%%%%%%%%%%%%%%%%%%%%%%%%%%%%%%%%%%%%%%

A $(\beta, \D)$-representation of a Gaussian integer $x = \sum_{k=0}^N x_k \beta^k$ is considered as an infinite string of the form $\cdots 0 0 0 x_{N} x_{N-1} \cdots x_1 x_0$ with $x_j \in \D$ for all $j \in \N$. We define a transducer, which starts reading the string from its right end to the left.

The transducer has 21 states, as shown in Figure~\ref{fig:Q-points}, which we summarize into the set~$Q \subset \Z[\beta]$: 
\begin{equation}\label{states}
  Q = \{0, \pm 1, \pm \ii, \pm 1 \pm \ii, \pm 2, \pm 2\ii, \pm 1 \pm 2\ii, \pm 2 \pm \ii \} \, . 
\end{equation}

%%%%%%%%%%%%%%%%%%%%%%%%%%%%%%%%%%%%%%%%%%%%%%%%%%
\begin{figure}
    \centering
    \begin{tikzpicture}[scale=0.6]
        \draw[gray,thin] (-2.5,-2.5) grid (2.5,2.5);
        \draw[<->,ultra thick] (-2.5,0) -- (2.5,0) node[right]{Re};
        \draw[<->,ultra thick] (0,-2.5) -- (0,2.5) node[above]{Im};
        % [0]
        \filldraw[blue] (0,0) circle (4pt);
        % [1]
        \filldraw[blue] (1,0) circle (4pt);
        \filldraw[blue] (-1,0) circle (4pt);
        \filldraw[blue] (0,1) circle (4pt);
        \filldraw[blue] (0,-1) circle (4pt);
        % [1+\ii]
        \filldraw[blue] (1,1) circle (4pt);
        \filldraw[blue] (-1,1) circle (4pt);
        \filldraw[blue] (1,-1) circle (4pt);
        \filldraw[blue] (-1,-1) circle (4pt);
        % [2]
        \filldraw[blue] (2,0) circle (4pt);
        \filldraw[blue] (-2,0) circle (4pt);
        \filldraw[blue] (0,2) circle (4pt);
        \filldraw[blue] (0,-2) circle (4pt);
        % [2+\ii]
        \filldraw[blue] (1,2) circle (4pt);
        \filldraw[blue] (2,1) circle (4pt);
        \filldraw[blue] (1,-2) circle (4pt);
        \filldraw[blue] (2,-1) circle (4pt);
        \filldraw[blue] (-1,2) circle (4pt);
        \filldraw[blue] (-2,1) circle (4pt);
        \filldraw[blue] (-1,-2) circle (4pt);
        \filldraw[blue] (-2,-1) circle (4pt);
    \end{tikzpicture}
    \caption{The set $Q$ of all 21~states of the transducer for converting any $(\beta, \D)$-representation into $3$-NAF representation, with $\beta = \ii-1$ and $\D = \{0, \pm 1, \pm \ii\}$.}
    \label{fig:Q-points}
\end{figure}
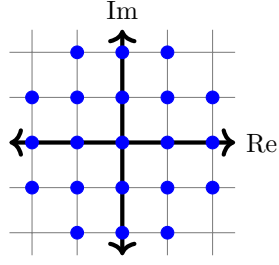
%%%%%%%%%%%%%%%%%%%%%%%%%%%%%%%%%%%%%%%%%%%%%%%%%%

Obviously, the set of states~$Q$ has numerous convenient symmetries: it is invariant under multiplication by $\ii$ and under complex conjugation -- i.e., $\ii Q = Q = \overline{Q}$, and also $\overline{\beta} Q = \beta Q$. Construction of the automaton is based on the following property of the set of states~$Q$:

%%%%%%%%%%%%%%%%%%%%%%%%%%%%%%%%%%%%%%%%%%%%%%%%%%
\begin{lemma}\label{prepis}
Let $Q$ be the set of states as defined in~\eqref{states}, and let $w \in Q + \D$ with $w \notin \beta Q$. Then, for every pair $b,c \in \D $, there exists a unique pair $d \in \D \setminus\{0\}$ and $q' \in Q$ such that   
$$w + b\beta + c\beta^2 = d+ \beta^3\, q'.$$
\end{lemma}
%%%%%%%%%%%%%%%%%%%%%%%%%%%%%%%%%%%%%%%%%%%%%%%%%%
\begin{proof}
By inspection of all pairs $b, c \in \D$.

\end{proof}
%%%%%%%%%%%%%%%%%%%%%%%%%%%%%%%%%%%%%%%%%%%%%%%%%%

%%%%%%%%%%%%%%%%%%%%%%%%%%%%%%%%%%%%%%%%%%%%%%%%%%
\begin{definition}\label{D:transducer}

The transducer converting $(\beta, \D)$-representation to $3$-NAF is defined as follows:

The initial state of the transducer is $q_0 = 0$. 

If the transducer is in a state $q \in Q$ and reads on input the digit $a \in \D$ such that $a + q = q' \beta \in \beta Q$, then it moves to the state~$q'$ and writes the digit~$0$ on output.

If $a + q \notin \beta Q$, then the transducer reads additional two digits $cb$ on input. By Lemma \ref{prepis}, the number $q + a + b \beta + c \beta^2 = q' \beta^3 + d$ for some state $q' \in Q$ and some digit $d \in \{\pm 1, \pm \ii \}$. In this case, the transducer moves to the state~$q'$ and writes on output the triplet of digits~$00d$. 

\end{definition}
%%%%%%%%%%%%%%%%%%%%%%%%%%%%%%%%%%%%%%%%%%%%%%%%%%

%%%%%%%%%%%%%%%%%%%%%%%%%%%%%%%%%%%%%%%%%%%%%%%%%%
\begin{theorem}\label{th:transducer}
Let $\D = \{0, \pm 1,\pm \ii \}$ and $\beta = \ii - 1$. The transducer described in Definition~\ref{D:transducer} assigns to the input $\cdots 0 0 0 x_{N} x_{N-1} \cdots x_1 x_0$ with $N \in \N$ and $x_k \in \D$ for all $k \in \N$ the $3$-NAF representation of the number $x = \sum_{k=0}^N x_k \beta^k$.  
\end{theorem}
%%%%%%%%%%%%%%%%%%%%%%%%%%%%%%%%%%%%%%%%%%%%%%%%%%
\begin{proof} 
Let us show that the transducer produces a string $\cdots 0 0 0 y_M y_{M-1} \cdots y_1 y_0$ of digits $y_j \in \D$, which is the $3$-NAF representation of the number $x = \sum_{k=0}^N x_k \beta^k$.   

The construction of the transducer implies that, in each state we visit during the calculation, the number of digits read on the input equals the number of digits written on the output - i.e., either 1~digit or 3~digits. We show that, if $q_n \in Q$ is the state after reading $n$~input digits and writing $n$~output digits, then   
\begin{equation}\label{mezikrok}
x = \sum_{k=0}^{n-1} y_k \beta^k + q_n \beta^n + \sum_{k=n}^N x_k \beta^k \, ,
\end{equation}
where $y_{n-1} y_{n-2} \cdots y_0$ is the string written on output (so far). Moreover, for each $k \in \N$, $k < n$, the string satisfies  
\begin{equation}\label{mezikrok2}
\text{if}\ \ y_k\neq 0, \ \text{then} \ k\leq n-3\ \ \text{and} \ \ y_{k+1} = y_{k+2} = 0 \, .
\end{equation}

We proceed by induction on~$n$. For $n=0$, there is nothing to prove, as the initial state is $q_0 = 0$, and thus $x = q_0 + \sum_{k=0}^N x_k \beta^k$. Next, let us assume that the claims \eqref{mezikrok} and~\eqref{mezikrok2} are satisfied for $n \in \N$. The transducer is in state~$q_n$ and reads the next new input digit~$x_n$. 

If $q_n + x_n \in \beta Q$, then, by definition of the transducer, the new state $q_{n+1} \in Q$ satisfies $q_{n+1} \beta = x_{n} + q_n$ and the new output digit is $y_n = 0$. Hence
$$x = \sum_{k=0}^{n-1} y_k \beta^k + \underbrace{ (q_n + x_n)}_{q_{n+1} \beta} \beta^n + \sum_{k=n+1}^N x_k \beta^k = \sum_{k=0}^{n-1} y_k \beta^k + 0 \cdot \beta^{n} + q_{n+1} \beta ^{n+1} + \sum_{k=n+1}^N x_k \beta^k \, .$$
Thus the equation~\eqref{mezikrok} is satisfied for the new state~$q_{n+1}$ and the output digits $0 y_{n-1} y_{n-2} \cdots y_0$. Clearly, the condition~\eqref{mezikrok2} is satisfied as well. 

If $q_n + x_n \notin \beta Q$, the transducer reads the next two digits $x_{n+2} x_{n+1}$. Then it moves to the state $q_{n+3} \in Q$ and writes output digits $y_n \in \{ \pm 1, \pm \ii \}$, $y_{n+1} = 0$, $y_{n+2} = 0$ such that $q_n + x_{n} + x_{n+1} \beta + x_{n+2} \beta^2 = q_{n+3} \beta^3 + y_{n}$. Therefore,
$$x = \sum_{k=0}^{n-1} y_k \beta^k + \underbrace{(q_n + x_n + x_{n+1} \beta + x_{n+2} \beta^2)}_{q_{n+3} \beta^3 + y_{n}} \beta^n + \sum_{k=n+3}^N x_k \beta^k =$$
$$= \sum_{k=0}^{n-1} y_k \beta^k + y_{n} \beta^{n} + 0 \cdot \beta^{n+1} + 0 \cdot \beta^{n+2}+ q_{n+3} \beta ^{n+3} + \sum_{k=n+3}^N x_k \beta^k \, .$$
The new string on output $00 y_{n} y_{n-1} \cdots y_0$ obviously satisfies both \eqref{mezikrok} and~\eqref{mezikrok2}.

It remains to show that, starting from some $M \in \N$, the digits $y_n$ on output are zero for each $n > M$. It suffices to prove that there exists an index $M > N$ such that the transducer gets into the state $q_M = 0$. If that is the case, then all the next states must be equal to zero, as  $x_j = 0$, and only zeros are written on the output $y_j =0$, on all positions $j > M$.

Assume that $n$~digits, for $n \geq N$, are already written on output and the transducer is in a state $q \neq 0$. We show that the new state $q' \in Q$, to which we get from the initial state $q$, fulfils $|q'| < |q|$. For $n > N$, we read $a = 0$ on input, and thus $w = q + a = q \in Q$. Let us discuss two cases: 
\begin{itemize}
    \item If $q = q' \beta$ for some state $q' \in Q$, then the output digit is $y = 0$, and the new state $q'$ satisfies $|q'| = |\beta|^{-1} |q| = \frac{1}{\sqrt{2}} |q| < |q|$.
    \item If $q \notin \beta Q$, then additional two zeros are read on the input, we move to the new state $q' \in Q$, and we write on output the triplet of digits $00d$ such that $q = q' \beta^3 + d$, where $d \in \{ \pm 1, \pm \ii \}$ and $q' \in Q$. The relation $q \notin \beta Q$ implies that
    \begin{itemize}
        \item either $q \in \{ \pm 1, \pm \ii \}$, and then $d = q$ and $q' = 0$, so clearly $|q'| = 0 < 1 = |q|$;
        \item otherwise $q \in \{ \pm 1 \pm 2\ii, \pm 2 \pm \ii \}$, so $|q| = \sqrt{5}$ and $|q - d| \leq \sqrt{10}$ for any $d \in \D$, therefore $|q'| = |\beta|^{-3} |q - d| \leq \frac{1}{2\sqrt{2}}\sqrt{10} = \frac{1}{2}\sqrt{5} = \frac{1}{2}|q| < |q|$.
    \end{itemize}
\end{itemize}
Since the number of states $q \in Q$ is finite, it is clear that, after a finite number of steps, the transducer reaches the state $q_M = 0$ for some $M > N$.

In a summary, the transducer produces the output string of the form $\cdots 0 0 0 y_M y_{M-1} \cdots y_1 y_0$ for some $M > N$, with $y_j \in \D$ for each index~$j$. Fulfilment of the properties \eqref{mezikrok} and~\eqref{mezikrok2} guarantees that $y_M y_{M-1} \cdots y_1 y_0$ is the $3$-NAF representation of~$x$.   
\end{proof}
%%%%%%%%%%%%%%%%%%%%%%%%%%%%%%%%%%%%%%%%%%%%%%%%%%

%%%%%%%%%%%%%%%%%%%%%%%%%%%%%%%%%%%%%%%%%%%%%%%%%%%%%%%%%%%%%%%%%%%%%%%%%%%%%%%%%%%%%%%%%%%%%%%%%%%%
\smallskip
\
\subsection{Transducer  as an Oriented Graph}\label{sub:graphG}
\
\medskip
%%%%%%%%%%%%%%%%%%%%%%%%%%%%%%%%%%%%%%%%%%%%%%%%%%%%%%%%%%%%%%%%%%%%%%%%%%%%%%%%%%%%%%%%%%%%%%%%%%%%

The finite transducer introduced in Definition~\ref{D:transducer} is usually interpreted as an oriented graph with labeled edges. 

%%%%%%%%%%%%%%%%%%%%%%%%%%%%%%%%%%%%%%%%%%%%%%%%%%
\begin{definition}\label{D:graph-G}
Let $\beta = \ii - 1$ and $\D = \{0, \pm 1, \pm \ii\}$. We denote by $G = (Q, E)$ the graph representing the transducer converting  $(\beta, \D)$-representations to $3$-NAF representations. The set~$Q$ of states of the transducer stands for the set of vertices of~$G$, while the state $q = 0 \in Q$ is the initial vertex. The set~$E$ of edges is defined as follows: \begin{enumerate}
    \item If $a + q = q' \beta \in \beta Q$, then $e = (q, q') \in E$ and the label of the edge~$e$ is $a\,|\,0$. The digits $a$ and~$0$ are denoted $In(e)$ and $Out(e)$, respectively, representing the input and the output. 
    \item If $c \beta^2 + b \beta + a + q = q' \beta^3 + d \in \beta^3 Q + \{\pm 1, \pm \ii\}$, then $e = (q, q') \in E$ and the label of the edge~$e$ is $cba\,|\,00d$. The triplets $cba$ and $00d$ are denoted $In(e)$ and $Out(e)$, respectively, again representing the input and the output. 
\end{enumerate}
An oriented path~$P$ of length~$\ell$ in the graph~$G$ is a sequence $q_0 e_1 q_1 e_2 q_2 e_3 q_3 \cdots e_{\ell} q_{\ell}$ such that $e_k = (q_{k-1}, q_k) \in E$ for each $k = 1, 2, \ldots, \ell$. 
\end{definition}
%%%%%%%%%%%%%%%%%%%%%%%%%%%%%%%%%%%%%%%%%%%%%%%%%%

%%%%%%%%%%%%%%%%%%%%%%%%%%%%%%%%%%%%%%%%%%%%%%%%%%
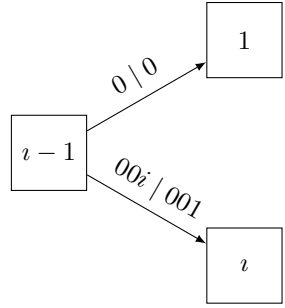
\begin{figure}
    \centering
    \begin{tikzpicture}[scale=1.5]
        \node[shape=rectangle,draw=black,minimum size=10mm] (i) at (0,0) {\(\ii \)};
        \node[shape=rectangle,draw=black,minimum size=10mm] (one) at (0,2) {\(1\)};
        \node[shape=rectangle,draw=black,minimum size=10mm] (i-minus-one) at (-{sqrt(3)},1) {\(\ii -1\)};
        \path[-latex] (i-minus-one) edge node[above,sloped] {$0 \mid 0$} (one);
        \path[-latex] (i-minus-one) edge node[above,sloped] {$00i \mid 001$} (i);
    \end{tikzpicture}
    \caption{Example of edges and their labels in graph~$G$ from Definition~\ref{D:graph-G} representing the transducer converting general $(\beta, \D)$-representations into $3$-NAF representation, with $\beta = \ii - 1$ and $\D = \{0, \pm 1, \pm \ii \}$.}\label{Fig:edgesInG} 
\end{figure}
%%%%%%%%%%%%%%%%%%%%%%%%%%%%%%%%%%%%%%%%%%%%%%%%%%

%%%%%%%%%%%%%%%%%%%%%%%%%%%%%%%%%%%%%%%%%%%%%%%%%%
\begin{remark}\label{re:cestaG}
Transformation of a general $(\beta, \D)$-representation $x_{N} x_{N-1} \cdots x_1 x_0$ of $x \in \Z[\ii]$ into its $3$-NAF representation corresponds to an oriented path $q_0 e_1 q_1 e_2 q_2 e_3 q_3 \cdots e_{\ell} q_{\ell}$ in $G$, where 
\begin{itemize}
    \item $q_0 = 0$ and $q_{\ell} = 0$;
    \item $In(e_\ell) In(e_{\ell -1}) \cdots In(e_{0})$ coincides with the $(\beta, \D)$-representation $x_N \cdots x_1 x_0$ of~$x$ (up to possible leading zeros);   
    \item $Out(e_\ell) Out(e_{\ell -1}) \cdots Out(e_{0})$ is the $3$-NAF representation of~$x$ in $(\beta, \D)$.
\end{itemize}
On the other hand, any oriented path in~$G$ starting and ending in~$q_0 = 0$ corresponds to transformation of some representation of a Gaussian integer into its $3$-NAF representation.  
\end{remark}
%%%%%%%%%%%%%%%%%%%%%%%%%%%%%%%%%%%%%%%%%%%%%%%%%%

Note that the usual notation of a path in a graph is from left to right - i.e., starting on the left with the initial vertex and ending on the right with the final vertex of the path, which is the opposite to the direction how representations of numbers are read by the transducer - i.e., from the right to the left end of the representation.

Let us recall that the set of states~$Q$ and the digit set~$\D$ are invariant under multiplication by~$\ii$ and under complex conjugation. Moreover, $\overline{\beta} = \ii \beta $. These facts imply a number of symmetries in the graph~$G$. 

%%%%%%%%%%%%%%%%%%%%%%%%%%%%%%%%%%%%%%%%%%%%%%%%%%
\begin{lemma}\label{le:HranoveSym}
Let $G = (Q, E)$ be the graph described in Definition~\ref{D:graph-G}. Let $e = (q, q') \in E$ have the label $a\,|\,0$, and let $f = (p, p') \in E$ have the label $cba\,|\,00d$. Then
\begin{enumerate}
    \item $(\ii q, \ii q') \in E$ and its label is $(\ii a)\,|\,0$; 
    \item $(\overline{q}, \ii \overline{q'}) \in E$ and its label is $\overline{a}\,|\,0$;
    \item $(\ii p, \ii p') \in E$ and its label is $(\ii c)(\ii b)(\ii a)\,|\,00(\ii d)$;
    \item $(\overline{p}, \overline{\ii p'}) \in E$ and its label is $(-\overline{c})(\ii \overline{b}) \overline{a}\,|\,00\overline{d}$.
\end{enumerate}    
\end{lemma}
%%%%%%%%%%%%%%%%%%%%%%%%%%%%%%%%%%%%%%%%%%%%%%%%%%
\begin{proof}
The fact that $e = (q, q') \in E$ has the label $a\,|\,0$ means that $a + q = q' \beta \in \beta Q$. By multiplication with~$\ii$ and by complex conjugation, and due to $\overline{\beta} = \ii \beta$, we get
$$ \text{$\ii a + \ii q = \ii q' \beta $ \quad and \quad $\overline{a} + \overline{q} = \overline{q'} \overline{\beta} = \ii \overline{q'} \beta$ ,}$$
which proves the claims (1) and~(2).

The fact that $f = (p, p') \in E$ has the label $cba\,|\,00d$ means that $c \beta^2 + b \beta + a + p = p' \beta^3 + d \in \beta^3 Q + \{\pm 1, \pm \ii\}$. Hence
$$ \ii c \beta^2 + \ii b \beta + \ii a + \ii p = \ii p' \beta^3 + \ii d \quad \text{and} \quad \overline{c} \overline{\beta^2} + \overline{b} \overline{\beta} + \overline{a} +\overline{p} = -\overline{c} \beta^2 + \ii \overline{b} \beta + \overline{a} + \overline{p} = \overline{p'} \overline{\beta}^3 + \overline{d} = \overline{\ii p'} \beta^3 + \overline{d} ,$$
which proves the claims (3) and~(4). 
\end{proof}
%%%%%%%%%%%%%%%%%%%%%%%%%%%%%%%%%%%%%%%%%%%%%%%%%%

The symmetries of the transducer, as shown above, enable to simplify the description of its graph~$G$ and help to understand its structure.

%%%%%%%%%%%%%%%%%%%%%%%%%%%%%%%%%%%%%%%%%%%%%%%%%%
\begin{definition}
Let $p, q \in Q$ be two vertices in the graph~$G$ representing two states of the transducer. We write
$$ p \sim q \text{\qquad if \  \ either}  \ \  \ p = \ii^k q \quad \text{or} \quad p = \ii^k\, \overline{q} \quad \text{\ for some } k \in \Z .$$    
\end{definition}

%%%%%%%%%%%%%%%%%%%%%%%%%%%%%%%%%%%%%%%%%%%%%%%%%%

Obviously, the relation~$\sim$ is an equivalence on~$Q$. It is easy to see that there exist 5 classes of the equivalence, namely
\begin{equation}\label{E:5EqClasses}
    \begin{split}
        [0] & = \{0\} \\
        [1] & = \{+1, +\ii, -1, -\ii\} \\
        [1+\ii] & = \{+1+\ii, -1+\ii, -1-\ii, +1-\ii\} \\
        [2] & = \{+2, +2\ii, -2, -2\ii\} \\
        [2+\ii] & = \{+2+\ii, -2+\ii, -2-\ii, +2-\ii, +1+2\ii, -1+2\ii, -1-2\ii, +1-2\ii\}
    \end{split}
\end{equation}

%%%%%%%%%%%%%%%%%%%%%%%%%%%%%%%%%%%%%%%%%%%%%%%%%%%%%%%%%%%%%%%%%%%%%%%%%%%%%%%%%%%%%%%%%%%%%%%%%%%%
\smallskip
\
\subsection{Restriction of the Transducer to Accept Only Optimal Representations}
\
\medskip
%%%%%%%%%%%%%%%%%%%%%%%%%%%%%%%%%%%%%%%%%%%%%%%%%%%%%%%%%%%%%%%%%%%%%%%%%%%%%%%%%%%%%%%%%%%%%%%%%%%%

Let us recall that a $(\beta, \D)$-representation of a Gaussian integer $x \in \Z[\ii]$ is called optimal, if the number of non-zero digits in the representation equals the number of non-zero digits in the $3$-NAF representation of~$x$.

%%%%%%%%%%%%%%%%%%%%%%%%%%%%%%%%%%%%%%%%%%%%%%%%%%
\begin{example}\label{ex:priklad2}
The string $1 0 0 \underline{\imath}$ is the $3$-NAF representation of the Gaussian integer $x = 2 + \imath$. It is easy to check that the strings $\imath 0 \imath$ and $\underline{\imath} 1$ also represent the same number~$x$. All the three representations of~$x$ are optimal, as they have the minimal Hamming weight~$3$. 
\end{example}
%%%%%%%%%%%%%%%%%%%%%%%%%%%%%%%%%%%%%%%%%%%%%%%%%%

In this section, we introduce a subgraph~$\G$ of the graph $G$ presented in Definition~\ref{D:graph-G}, which shall transform any optimal representation of an arbitrary Gaussian integer $x \in \Z[\ii]$ into $3$-NAF representation, and, at the same time, none of the non-optimal representations is accepted by this subgraph~$\G$.

%%%%%%%%%%%%%%%%%%%%%%%%%%%%%%%%%%%%%%%%%%%%%%%%%%
\begin{definition}
Let $e = (q, q') \in E$ be an edge of the graph ~$G$ with label $In(e)\,|\,Out(e)$. The number of non-zero digits in $In(e)$ minus the number of non-zero digits in $Out(e)$ is called the weight of~$e$, and denoted by~$w(e)$. The weight~$w(P)$ of an oriented path $P$ in~$G$ is defined as the sum of weights of all edges in~$P$. 
\end{definition}
%%%%%%%%%%%%%%%%%%%%%%%%%%%%%%%%%%%%%%%%%%%%%%%%%%

%%%%%%%%%%%%%%%%%%%%%%%%%%%%%%%%%%%%%%%%%%%%%%%%%%
\begin{remark}\label{re:vahy}
Let $In(e)\,|\,Out(e)$ be the label of an edge $e = (q, q')$ in graph~$G$. 
\begin{enumerate}
   \item If $In(e)\,|\,Out(e) = a\,|\,0$, then obviously $w(e) = |a| \in \{0, 1\}$.
   \item If $In(e)\,|\,Out(e) = cba\,|\,00d$, then $w(e) \in \{-1, 0, 1, 2\}$, as $d \in \{\pm 1, \pm \ii\}$ and $a, b, c \in \D$. Moreover,
    \begin{itemize}
        \item $w(e) = -1$ only if $a = b = c = 0$;
        \item $w(e) = 0$ if exactly two digits among $a, b, c$ are zero; 
        \item $w(e) = 1$ if exactly one  digit among $a, b, c$ is zero; and
        \item $w(e) = 2$ if $a, b, c \neq 0$.
    \end{itemize}
\end{enumerate}
\end{remark}
%%%%%%%%%%%%%%%%%%%%%%%%%%%%%%%%%%%%%%%%%%%%%%%%%%

The fact that the $3$-NAF representation of a Gaussian integer~$x$ contains the minimal number of non-zero digits among all representations of~$x$ can be reformulated as a property of the graph~$G$.   

%%%%%%%%%%%%%%%%%%%%%%%%%%%%%%%%%%%%%%%%%%%%%%%%%%
\begin{lemma}\label{le:nezaporneCesty}
Let $P$ be an oriented path in the graph~$G$ such that $P$~starts and ends in the vertex $q_0 = 0$. Then the weight of~$P$ is non-negative.   
\end{lemma}
%%%%%%%%%%%%%%%%%%%%%%%%%%%%%%%%%%%%%%%%%%%%%%%%%%
\begin{proof}
By Remark~\ref{re:cestaG}, every path $P$ in~$G$ which starts and ends in $q_0 = 0$ corresponds to transformation of a general $(\beta, \D)$-representation of a Gaussian integer to its $3$-NAF representation. By Theorem \ref{thm:MinimalHamming}, the weight of~$P$ is non-negative.
\end{proof}
%%%%%%%%%%%%%%%%%%%%%%%%%%%%%%%%%%%%%%%%%%%%%%%%%%

%%%%%%%%%%%%%%%%%%%%%%%%%%%%%%%%%%%%%%%%%%%%%%%%%%
\begin{lemma}\label{le:stejneVahy}
Let $P$ be a path in the graph~$G$ starting in vertex~$u$ and ending in vertex~$v$, and let $u' \in [u]$. Then there exists a path $P'$ in~$G$ of the same length and the same weight as~$P$, such that $P'$~starts in~$u'$ and ends in some $v' \in [v]$.    
\end{lemma}
%%%%%%%%%%%%%%%%%%%%%%%%%%%%%%%%%%%%%%%%%%%%%%%%%%
\begin{proof}
It follows from Lemma~\ref{le:HranoveSym}, as any transformation of an edge by means of complex conjugation or multiplication with~$\ii$ preserves the equivalence classes of vertices forming the edge, and also preserves the weight of the edge.
\end{proof}
%%%%%%%%%%%%%%%%%%%%%%%%%%%%%%%%%%%%%%%%%%%%%%%%%%

%%%%%%%%%%%%%%%%%%%%%%%%%%%%%%%%%%%%%%%%%%%%%%%%%%
\begin{lemma}\label{le:eliminace}
Let $q_0 e_1 q_1 e_2 q_2 e_3 q_3 \cdots e_{\ell} q_{\ell}$ be a path in the graph~$G$ which transforms an optimal $(\beta, \D)$-representation of $x \in \Z[\ii]$ into the $3$-NAF representation. Let $j, n \in \N, 0 \leq j < n \leq \ell$, $u \in [q_j]$ and $v \in [q_n]$.
Then no path (of any length) in~$G$ starting in~$u$ and ending in~$v$ has weight strictly smaller than the weight of the path $q_{j} e_{j+1} q_{j+1} e_{j+2} q_{j+2} \cdots e_n q_n$.  
\end{lemma}
%%%%%%%%%%%%%%%%%%%%%%%%%%%%%%%%%%%%%%%%%%%%%%%%%%
\begin{proof}
As the path $P = q_0 e_1 q_1 e_2 q_2 e_3 q_3 \cdots e_{\ell} q_{\ell}$ transforms an optimal representation to the $3$-NAF representation, its weight must be $w(P) = 0$, and the vertices $q_0 = 0$ and $q_{\ell} = 0$. Let us split the path~$P$ into three parts:
$$
P_1 = q_0 e_1 q_1 e_2 q_2 \cdots e_j q_j \, , \quad 
P_2 = q_{j} e_{j+1} q_{j+1} e_{j+2} q_{j+2} \cdots e_n q_n \, , \quad
P_3 = q_{n} e_{n+1} q_{n+1} e_{n+2} q_{n+2} \cdots e_{\ell} q_{\ell} . 
$$
Clearly, $w(P_1) + w(P_2) + w(P_3) = w(P) = 0$. 

Assume, contrary to what we aim to prove, that there exists a path~$P_2'$ starting in~$u$ and ending in~$v$ such that $w(P_2') < w(P_2)$.

By Lemma \ref{le:stejneVahy}, there exists a path $P''_2$ starting in $q_j \in [u]$ and ending in some vertex $v' \in [v] = [q_n]$ such that $w(P''_2) = w(P'_2)$. Due to the same lemma, there exists a path $P'_3$ starting in~$v'$ and ending in vertex~$q_l = 0$ such that $w(P'_3) = w(P_3)$.  

Since the ending vertex of~$P_1$ coincides with the starting vertex of~$P''_2$, and the ending vertex of~$P_2''$ coincides with the starting vertex of~$P_3'$, the concatenation of the three parts produces a new path $P_1 P''_2 P_3'$ starting and ending in vertex~$q_0 = q_l = 0$. The weight of this new path is $$w(P_1) + w(P''_2) + w(P_3') = w(P_1) + w(P'_2) + w(P_3) < w(P_1) + w(P_2) + w(P_3) = 0 \, ,$$
which is a contradiction with Lemma \ref{le:nezaporneCesty}.
\end{proof}
%%%%%%%%%%%%%%%%%%%%%%%%%%%%%%%%%%%%%%%%%%%%%%%%%%

%%%%%%%%%%%%%%%%%%%%%%%%%%%%%%%%%%%%%%%%%%%%%%%%%%
\begin{definition}
A path $P$ in the graph~$G$ which transforms an optimal representation of a Gaussian integer to its $3$-NAF representation is said to be optimal.
\end{definition}
%%%%%%%%%%%%%%%%%%%%%%%%%%%%%%%%%%%%%%%%%%%%%%%%%%

By application of Lemma~\ref{le:eliminace} onto $n = j+1$, we obtain a simple observation:  

%%%%%%%%%%%%%%%%%%%%%%%%%%%%%%%%%%%%%%%%%%%%%%%%%%
\begin{corollary}\label{le: optimHrany}
Let $e = (p, q)$ and $e' = (p', q')$ be two edges of graph~$G$ such that $p \sim p'$ and $q \sim q'$. If $w(e) < w(e')$, then no optimal path in~$G$ uses the edge~$e'$.\

Let $f = (p, p')$ be an edge of~$G$ such that $p \sim p'$ and $w(f) > 0$. Then no optimal path in~$G$ uses the edge~$f$.  
\end{corollary}
%%%%%%%%%%%%%%%%%%%%%%%%%%%%%%%%%%%%%%%%%%%%%%%%%%

In Lemma~\ref{le: kondenzovanyGraf1} we describe vertices of the graph~$G$ which can be visited by an optimal path. For this purpose, we introduce another graph, denoted by~$\Gamma$, operating with the equivalence classes~$[u]$ on the set of states~$Q$:

%%%%%%%%%%%%%%%%%%%%%%%%%%%%%%%%%%%%%%%%%%%%%%%%%%
\begin{definition}\label{D:graph-Gamma}
Let $\Gamma$ be a graph, whose vertices are the equivalence classes~$[u]$ on the set of states~$Q$, as defined in~\eqref{E:5EqClasses}. It means that $\Gamma$~has five vertices: $[0], [1], [1+\ii], [2], [2+\ii]$.
Two vertices $[u], [v]$ of~$\Gamma$ are connected with an edge $\eta = ([u], [v])$ according to the following rules: 
\begin{itemize}
    \item {\bf Case} $\mathbf{[u] \neq [v]}$: An edge $\eta = ([u], [v])$ with weight $w(\eta) = m$ belongs to the graph~$\Gamma$, if
        \begin{itemize}
            \item there exist $u' \in [u]$ and $v' \in [v]$ such that $e' = (u', v')$ is an edge in the graph~$G$ with weight $w(e') = m$, and
            \item there is no edge $e'' = (u'', v'')$ in the graph~$G$ with $u'' \in [u], v'' \in [v]$ and with weight $w(e'') < m$.
        \end{itemize}
    \item {\bf Case} $\mathbf{[u] = [v]}$: A loop $\eta = ([u], [u])$ with weight~$w(\eta) = 0$ belongs to the graph~$\Gamma$, if there exist $u', v' \in [u]$ such that $e' = (u', v')$ is an edge in the graph~$G$ with weight $w(e') = 0$. 
\end{itemize}
\end{definition}
%%%%%%%%%%%%%%%%%%%%%%%%%%%%%%%%%%%%%%%%%%%%%%%%%%

%%%%%%%%%%%%%%%%%%%%%%%%%%%%%%%%%%%%%%%%%%%%%%%%%%
\begin{lemma}\label{lem:kondenzovanyG}
The graph $\Gamma$ introduced in Definition~\ref{D:graph-Gamma} is the graph depicted in Figure~\ref{F:graph-Gamma-condensed}.
\end{lemma}
%%%%%%%%%%%%%%%%%%%%%%%%%%%%%%%%%%%%%%%%%%%%%%%%%%
\begin{proof}

Firstly, we need to show for every pair of vertices $[u]$ and~$[v]$ which are not connected by any edge in the graph in Figure~\ref{F:graph-Gamma-condensed}, that there exist no $u' \in [u]$ and $v' \in [v]$ connected by some edge in the graph~$G$.

\begin{itemize}
    
    \item There is no edge in~$\Gamma$ starting in $[u] \in \{[0], [1]\}$ and ending in $[v] \in \{[2], [2+\ii]\}$. If there existed any edge $e = (u', v')$ in~$G$ with $u' \in [u]$ and $v' \in [v]$, then either $u' + a = v' \beta$, or $u' + c \beta^2 + b \beta + a = v' \beta^3 + d$ for some $a, b, c \in \D$ and $d \in \{\pm 1, \pm \ii\}$. The first possibility gives $2 \geq |u' + a|= |v'| |\beta| \geq 2\sqrt{2}$ -- a contradiction. If the first possibility does not occur, then $|u' + a| = 1$, and hence $\sqrt{17} \geq |u' + c \beta^2 + b \beta + a| = |v' \beta^3 + d| \geq 5$ -- so again a contradiction.

    \item Only vertices $[0]$ and $[1]$ can have a loop. A vertex~$[u]$ has a loop in~$\Gamma$, only if there exists an edge $(u', v')$ of weight~$0$ in~$G$ with $u', v' \in [u]$:
    \begin{itemize}
        \medskip
        
        \item either with a label $0\,|\,0$, in this case $u' = v' \beta$, which implies $u' = 0$ due to $|u'| = |v'|$;

        \medskip
        
        \item or with a label $cba\,|\,00d$, where $u' + a \notin \beta Q$ and exactly one digit in~$cba$ being non-zero, in this case $u' + c \beta^2 + b \beta + a = v' \beta^3 + d$. We use again the fact that $|u'| = |v'|$.\
        
        If $a \neq 0$, then $b = c = 0$, consequently $|u'| + 1 \geq |u' + a| = |v' \beta^3 + d| \geq |v'| |\beta^3| - 1 = 2 \sqrt{2} |u'| - 1 \, ,$ and thus $|u'| \leq \frac{2}{2\sqrt{2}-1} < \sqrt{2}$.  Therefore, $u' \in [0]\cup [1]$.

        If $a = 0$, then $u' + a \notin \beta Q$ implies $u' \in [1] \cup [2+\ii]$. Due to $|c| + |b| = 1$, we get the inequality $|u'| + 2 \geq |u' + c \beta^2 + b \beta| = |v' \beta^3 + d| \geq |v'| |\beta^3| - 1 = 2 \sqrt{2} \,|u'| - 1 \, ,$ which forces $|u'| \leq \frac{3}{2\sqrt{2}-1} < \sqrt{5} = |2 + \ii|$, and hence $u' \notin [2+\ii]$.

    \end{itemize}
\end{itemize}

Secondly, we have to prove for any edge $\eta = ([u], [v])$ with label~$m$ in the graph depicted in Figure~\ref{F:graph-Gamma-condensed}, that there exists an edge $e = (u',v')$ in graph~$G$ with weight $w(e) = m$. All such edges are listed in Table~\ref{fig:kondenzG}. 

%%%%%%%%%%%%%%%%%%%%%%%%%%%%%%%%%%%%%%%%%%%%%%%%%%
\begin{table}[h]
    \centering
    \setlength{\tabcolsep}{3pt}
    \renewcommand{\arraystretch}{1.2}
    \begin{tabular}{|c|c|c|c|c|c|c|}
    \hline
    vertex $[u]$ & vertex $[v]$ & $u'\in [u]$ & $v'\in [v]$ & equation fulfilled by & label in $G$ & weight \\
    in $\Gamma$ & in $\Gamma$ & in $G$ & in $G$ & $e = (u',v')$ in $G$ &  $In(e)\,|\,Out(e)$ & $w(e)$ \\
    \hline \hline
    $[0]$ & $[0]$ & $0$ & $0$ & $u' = v' \beta$ & $0 \,|\, 0$ & $0$ \\
    \hline
    $[0]$ & $[1]$ & $0$ & $1$ & $u' + \ii\beta^2 + \ii = v' \beta^3 - \ii$ & $\ii 0 \ii \,|\, 0 0 \underline{\ii}$ & $1$ \\
    \hline
    $[0]$ & $[1+\ii]$ & $0$ & $-1+\ii$ & $u' - \ii\beta^2 - 1 = v' \beta^3 + 1$ & $\underline{\ii} 0 \underline{1} \,|\, 0 0 1$ & $1$ \\
    \hline
    $[1]$ & $[0]$ & $1$ & $0$ & $u' = v' \beta^3 + 1$ & $0 0 0 \,|\, 0 0 1$ & $-1$ \\
    \hline
    $[1]$ & $[1]$ & $1$ & $1$ & $u' - \beta^2 = v' \beta^3 - 1$ & $\underline{1} 0 0 \,|\, 0 0 \underline{1}$ & $0$ \\
    \hline
    $[1]$ & $[1+\ii]$ & $1$ & $1-\ii$ & $u' + \ii\beta^2 = v' \beta^3 - 1$ & $\ii 0 0 \,|\, 0 0 \underline{1}$ & $0$ \\
    \hline
    $[1+\ii]$ & $[0]$ & $1+\ii$ & $0$ & $u' - \ii = v' \beta^3 + 1$ & $0 0 \underline{\ii} \,|\, 0 0 1$ & $0$ \\
    \hline
    $[1+\ii]$ & $[1]$ & $-1+\ii$ & $1$ & $u' = v' \beta$ & $0 \,|\, 0$ & $0$ \\
    \hline
    $[1+\ii]$ & $[2]$ & $1+\ii$ & $2$ & $u' - \beta^2 - \ii\beta + 1 = v' \beta^3 - 1$ & $\underline{1} \underline{\ii} 1 \,|\, 0 0 \underline{1}$ & $2$ \\
    \hline
    $[1+\ii]$ & $[2+\ii]$ & $1+\ii$ & $2+\ii$ & $u' - \beta^2 - \ii\beta + \ii = v' \beta^3 - \ii$ & $\underline{1} \underline{\ii} \ii \,|\, 0 0 \underline{\ii}$ & $2$ \\
    \hline
    $[2]$ & $[0]$ & $2$ & $0$ & $u' - 1 = v' \beta^3 + 1$ & $0 0 \underline{1} \,|\, 0 0 1$ & $0$ \\
    \hline
    $[2]$ & $[1]$ & $2$ & $1$ & $u' + \ii = v' \beta^3 - \ii$ & $0 0 \ii \,|\, 0 0 \underline{\ii}$ & $0$ \\
    \hline
    $[2]$ & $[1+\ii]$ & $2\ii$ & $1+\ii$ & $u' + \ii = v' \beta^3 - \ii$ & $0 0 \ii \,|\, 0 0 \underline{\ii}$ & $0$ \\
    \hline
    $[2]$ & $[2+\ii]$ & $2\ii$ & $2+\ii$ & $u' - \beta^2 - \ii\beta + 1 = v' \beta^3 - \ii$ & $\underline{1} \underline{\ii} 1 \,|\, 0 0 \underline{\ii}$ & $2$ \\
    \hline
    $[2+\ii]$ & $[0]$ & $2+\ii$ & $0$ & $u' - \ii\beta^2 = v' \beta^3 + \ii$ & $\underline{\ii} 0 0 \,|\, 0 0 \ii$ & $0$ \\
    \hline
    $[2+\ii]$ & $[1]$ & $2+\ii$ & $1$ & $u'=  v' \beta^3-\ii$ & $0 0 0 \,|\, 0 0 \underline{\ii}$ & $-1$ \\
    \hline
    $[2+\ii]$ & $[1+\ii]$ & $2+\ii$ & $1-\ii$ & $u' - \beta = v' \beta^3 - 1$ & $0 \underline{1} 0 \,|\, 0 0 \underline{1}$ & $0$ \\
    \hline
    $[2+\ii]$ & $[2]$ & $2+\ii$ & $2$ & $u' - \beta^2 - \ii\beta = v' \beta^3 - 1$ & $\underline{1} \underline{\ii} 0 \,|\, 0 0 \underline{1}$ & $1$ \\
    \hline
    \end{tabular}
\smallskip
\caption{List of all edges of graph~$\Gamma$ introduced by Definition \ref{D:graph-Gamma}, for $\beta = \ii - 1$ and $\D = \{0, \pm 1, \pm \ii\}$.}\label{fig:kondenzG}
\end{table}
%%%%%%%%%%%%%%%%%%%%%%%%%%%%%%%%%%%%%%%%%%%%%%%%%%

Lastly, to complete the proof, we show for any pair $[u], [v]$ of vertices connected by an edge of weight~$m$ in~$\Gamma$ that no pair $u' \in [u]$ and $v' \in [v]$ is connected in~$G$ by any edge of weight smaller than~$m$. By Remark~\ref{re:vahy}, the weight of any edge in~$G$ belongs to $\{-1, 0,1,2\}$. 

\begin{itemize}

    \item  Only two edges in~$\Gamma$ have the weight~$-1$, namely $([1], [0])$ and $([2+\ii], [1])$. Indeed, the label of an edge of weight~$-1$ must necessarily have the form $000\,|\,00d$ with $|d|=1$. The relation $u' = v' \beta^3 + d$ implies $\sqrt{5} \geq |u'| = |v' \beta^3 + d| \geq |v'| |\beta^3| - 1 = 2\sqrt{2} |v'| - 1 \, ,$ and hence $|v'| \leq \frac{1 + \sqrt{5}}{2\sqrt{2}} < \sqrt{2}$, so $v' \in [0] \cup [1]$. If $v' = 0$, then the equation $u' = v' \beta^3 + d$ forces $u' \in [1]$. If $v' \in [1]$, then $u' = v' \beta^3 + d$ forces $u' \in [2+\ii]$.

    \item Each edge in~$\Gamma$ starting in $[u] \in [1+\ii]$ and ending in $[v] \in \{[2], [2+\ii]\}$ has weight~$2$. We show that any edge $e = (u', v')$ in~$G$ with $u' \in [u]$ and $v' \in [v']$ has a label of the form $cba\,|\,00d$ with all digits $c, b, a, d$ being non-zero. In order to deduce it, we use the fact that $|u'| = \sqrt{2}$ and $|v'| \geq 2$. Indeed, 
    \begin{itemize}
        
        \item $a\,|\,0$ cannot be the label of~$e$, since $u' + a = v' \beta$ implies $\sqrt{2} + 1 \geq |u' + a| = |v' \beta| \geq 2\sqrt{2}$, which is a contradiction;
        
        \item $cba\,|\,00d$ having one of the digits $c, b, a$ zero cannot be the label of~$e$, as $u' + c \beta^2 + b \beta + a = v' \beta^3 + d$ would imply $\sqrt{2} + \sqrt{10} \geq |u'| + |c \beta^2 + b \beta + a| \geq |u' + c \beta^2 + b \beta + a| = |v' \beta^3 + d| \geq 4 \sqrt{2} - 1$, so again a contradiction.
   
    \end{itemize}

    \item No edge in~$G$ of weight smaller than~$2$ starts in $u' \in [2]$ and ends in $v' \in [2+\ii]$. Considering a label $a\,|\,0$, we obtain the relation $3 \geq |u' + a| = |v' \beta| = \sqrt{10}$, so a contradiction. Considering a label $cba\,|\,00d$ with at most two digits among $c, b, a$ being non-zero, we obtain the relation $\sqrt{26} \geq |u' + c \beta^2 + b \beta + a| = |v' \beta^3 + d| \geq |v'| |\beta^3| - 1 = 2\sqrt{10} - 1$, again a contradiction.

    \item We show that no edge in~$G$ of weight~$0$ starts in $u' = 0$ and ends in $v' \in [1] \cup [1+\ii]$.  If such an edge existed, then by Remark~\ref{re:vahy}, its label is $cba\,|\,00d$ and at most one digit among $c, b, a$ is non-zero. The relation $u' + c \beta^2 + b \beta + a = v' \beta^3 + d$ implies $2 \geq |0 + c \beta^2 + b \beta + a| = |v'\beta^3+d| \geq \sqrt{5}$, a contradiction.

    \item The edge $([2+\ii], [2])$ has weight~$1$ in~$\Gamma$. We show that no edge in~$G$ of weight~$0$ starts in $u' \in [2+\ii]$ and ends in $v' \in [2]$. If such an edge existed, then its label could not be $a\,|\,0$, since necessarily $a = 0$, but $\sqrt{5} = |2+\ii| = |u' + a| \neq |v' \beta| = 2 \sqrt{2}$. Considering a label $cba\,|\,00d$ with at most one digit among $c, b, a$ being non-zero, we obtain the relation $\sqrt{17} \geq |u' + c \beta^2 + b \beta + a| = |v' \beta^3 + d| \geq |v'| |\beta^3| - 1 \geq 4\sqrt{2} - 1$, a contradiction.

\end{itemize}
\end{proof}
%%%%%%%%%%%%%%%%%%%%%%%%%%%%%%%%%%%%%%%%%%%%%%%%%%

%%%%%%%%%%%%%%%%%%%%%%%%%%%%%%%%%%%%%%%%%%%%%%%%%%
\begin{figure}
    \centering
    \begin{tikzpicture}[scale=2]
        \node[shape=circle,draw=black,minimum size=13mm] (zero) at (0,0) {\([0]\)};
        \node[shape=circle,draw=black,minimum size=13mm] (one) at (2,0) {\([1]\)};
        \node[shape=circle,draw=black,minimum size=13mm] (two) at (4,0) {\([2]\)};
        \node[shape=circle,draw=black,minimum size=13mm] (i-plus-one) at (1,-{sqrt(3)}) {\([1+\ii]\)};
        \node[shape=circle,draw=black,minimum size=13mm] (i-plus-two) at (3,-{sqrt(3)}) {\([2+\ii]\)};
        \path[-latex]
            (zero) edge[bend right=10] node[left] {$+1$} (i-plus-one)
            (i-plus-one) edge[bend right=10] node[right] {$0$} (zero)
            (zero) edge[bend right=10] node[below] {$+1$} (one)
            (one) edge[bend right=10] node[above] {$-1$} (zero)
            (i-plus-one) edge[bend right=10] node[right] {$0$} (one)
            (one) edge[bend right=10] node[left] {$0$} (i-plus-one)
            (zero) edge[>=latex,loop left] node[left] {$0$} ()
            (one) edge[>=latex,loop above] node[below left=7pt] {$0$} ()
            (two) edge node[above] {$0$} (one)
            (two) edge[bend right=10] node[below left=6pt] {$+2$} (i-plus-two)
            (i-plus-two) edge[bend right=10] node[right] {$+1$} (two)
            (i-plus-two) edge node[above right=-5pt] {$-1$} (one)
            (i-plus-one) edge[bend right=10] node[below] {$+2$} (i-plus-two)
            (i-plus-two) edge[bend right=10] node[above=-2pt] {$0$} (i-plus-one)
            (i-plus-one) edge[bend right=10] node[below left=6pt] {$+2$} (two)
            (two) edge[bend right=10] node[above right=6pt] {$0$} (i-plus-one)
            (two) edge[bend right=35] node[right=30pt] {$0$} (zero)
            (i-plus-two) edge[bend left=95] node[below] {$0$} (zero)
        ;
    \end{tikzpicture}
    \caption{Graph~$\Gamma$ for $\beta = \ii - 1$ and $\D = \{0, \pm 1, \pm \ii\}$ introduced in Definition~\ref{D:graph-Gamma}, and fully described in the proof of Lemma~\ref{lem:kondenzovanyG}. Each edge~$e = ([u], [v])$ of~$\Gamma$ is marked with its weight~$w(e) \in \{-1, 0, 1, 2\}$.}
\label{F:graph-Gamma-condensed}
\end{figure}
%%%%%%%%%%%%%%%%%%%%%%%%%%%%%%%%%%%%%%%%%%%%%%%%%%

%%%%%%%%%%%%%%%%%%%%%%%%%%%%%%%%%%%%%%%%%%%%%%%%%%
\begin{lemma}\label{le: kondenzovanyGraf1}
Let $P=q_0 e_1 q_1 e_2 q_2 e_3 q_3 \cdots e_{\ell} q_{\ell}$ be an optimal path in graph~$G$. Then $q_k \in [0] \cup [1] \cup [1+\ii]$ for every $k \in \{0, 1, \ldots, \ell\}$. Moreover, if $q_k \in [1+\ii]$, then $q_{k+1} \neq 0$.   
\end{lemma}
%%%%%%%%%%%%%%%%%%%%%%%%%%%%%%%%%%%%%%%%%%%%%%%%%%
\begin{proof}
Since the path~$P$ is optimal, Corollary~\ref{le: optimHrany} implies for each $k = 1, 2, \ldots, \ell$ that $([q_{k-1}], [q_k])$ is an edge in~$\Gamma$ with weight~$w(e_k)$. Hence, any optimal path~$P$ in~$G$ has a corresponding path in~$\Gamma$ with the same weight.

First part of the lemma is shown by contradiction. Assume that at least one vertex of the optimal path belongs to $[2] \cup [2+\ii]$. It means that $P$~has a subpath $P' = p_0 f_1 p_1 f_2 \cdots f_j p_j$, where $2 \leq j \leq \ell$, $p_0, p_j \in [0] \cup [1] \cup [1+\ii]$ and $p_1, p_2, \ldots, p_{j-1} \in [2] \cup [2+\ii]$. The form of the graph~$\Gamma$, as depicted in Figure~\ref{F:graph-Gamma-condensed},  implies the following:
\begin{itemize}
    \item $p_0 \in [1+\ii]$ and weight $w(f_1) = 2$;
    \item $w(f_k) \geq 1$ for each $k = 2, 3, \ldots, j-1$; 
    \item weight $w(f_j) \geq -1$. 
\end{itemize}
Consequently, $w(P') = \sum_{k=1}^j w(f_k) \geq 1$.

On the other hand, there exists a path of weight $0$ or~$-1$ starting in $p_0 \in [1+\ii]$ and ending in $p_j \in [0] \cup [1] \cup [1+\ii]$, which contradicts Lemma~\ref{le:eliminace}. 

To complete the proof, we observe that there is a path of length~$2$ starting in~$[1+\ii]$, ending in~$[0]$, and having weight~$-1$. Lemma~\ref{le:eliminace} implies that the edge $([1+\ii], [0])$ of weight~$0$ cannot be used in the optimal path.
\end{proof}
%%%%%%%%%%%%%%%%%%%%%%%%%%%%%%%%%%%%%%%%%%%%%%%%%%

%%%%%%%%%%%%%%%%%%%%%%%%%%%%%%%%%%%%%%%%%%%%%%%%%%
\begin{corollary}\label{co:zmensenyGraf}
Let $P = q_0 e_1 q_1 e_2 q_2 e_3 q_3 \cdots e_{\ell} q_{\ell}$ be a path in~$G$ with $q_0 = q_\ell = 0$ and let $\GG$ be the graph depicted on Figure~\ref{F:graph-Gamma-limited}. The path~$P$ is optimal if and only if, for every $k \in \{0, 1, \ldots, \ell\}$, the two properties are fulfilled:
$$q_k \in [0] \cup [1] \cup [1+\ii] \quad \text{and} \quad w(e_k) \text{\ equals the weight of the edge \ } w([q_{k-1}], [q_{k}]) \in \GG \, .$$
\end{corollary}
%%%%%%%%%%%%%%%%%%%%%%%%%%%%%%%%%%%%%%%%%%%%%%%%%%
\begin{proof}
It follows from the proof of Lemma~\ref{le: kondenzovanyGraf1} that optimal paths can use only vertices and edges from the subgraph~$\GG$. The opposite implication follows from the fact that all paths in~$\GG$, which start and end in~$[0]$, have weight~$0$, and thus are optimal.
\end{proof}
%%%%%%%%%%%%%%%%%%%%%%%%%%%%%%%%%%%%%%%%%%%%%%%%%%

%%%%%%%%%%%%%%%%%%%%%%%%%%%%%%%%%%%%%%%%%%%%%%%%%%
\begin{figure}
    \centering
    \begin{tikzpicture}[scale=1.5]
        \node[shape=circle,draw=black,minimum size=13mm] (zero) at (0,0) {\([0]\)};
        \node[shape=circle,draw=black,minimum size=13mm] (one) at (2,0) {\([1]\)};
        \node[shape=circle,draw=black,minimum size=13mm] (i-minus-one) at (1,-{sqrt(3)}) {\([1+\ii]\)};
        \path[-latex]
            (zero) edge node[left] {$+1$} (i-minus-one)
            (zero) edge[bend right=10] node[below] {$+1$} (one)
            (one) edge[bend right=10] node[above] {$-1$} (zero)
            (i-minus-one) edge[bend right=10] node[right] {$0$} (one)
            (one) edge[bend right=10] node[left] {$0$} (i-minus-one)
            (zero) edge[>=latex,loop left] node[left] {$0$} ()
            (one) edge[>=latex,loop right] node[right] {$0$} ()
        ;
    \end{tikzpicture}
    \caption{Graph $\GG$ -- a sub-graph of graph~$\Gamma$ (given on Figure~\ref{F:graph-Gamma-condensed}) for $\beta = \ii - 1$ and $\D = \{0, \pm 1, \pm \ii\}$, limited to optimal paths only, labeled with weights of its edges.}
    \label{F:graph-Gamma-limited}
\end{figure}
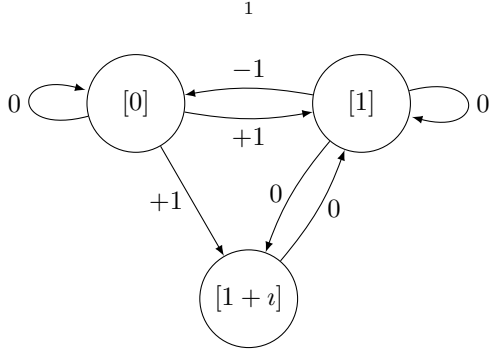
%%%%%%%%%%%%%%%%%%%%%%%%%%%%%%%%%%%%%%%%%%%%%%%%%%

%%%%%%%%%%%%%%%%%%%%%%%%%%%%%%%%%%%%%%%%%%%%%%%%%%%%%%%%%%%%%%%%%%%%%%%%%%%%%%%%%%%%%%%%%%%%%%%%%%%%
\smallskip
\
\subsection{Number of Optimal Representations}
\
\medskip
%%%%%%%%%%%%%%%%%%%%%%%%%%%%%%%%%%%%%%%%%%%%%%%%%%%%%%%%%%%%%%%%%%%%%%%%%%%%%%%%%%%%%%%%%%%%%%%%%%%%

In this subsection, we create a matrix formula, which assigns the number of optimal representations  to any Gaussian integer $x \in \Z[\ii]$.  In fact, Proposition \ref{prop:maticovy vzorecek} says that the mapping which assigns to the $3$-NAF representation of~$x$ the number of  its optimal representations is a recognisable series over the alphabet $\mathcal{C} =\{0\}\cup\{00d: d \in \{1,-1,\ii,-\ii\}\}$.   

We recall the notion of the recognisable series as introduced  in \cite{BeRe2010}.  Let  $\mathbb{F}$ be a field and $\mathcal{C}$ be a finite alphabet.  A mapping $f : \mathcal{C}^* \mapsto \mathbb{F}$ is said to be a recognisable
series if there are a non-negative integer $k$, a~family $(M_c), c \in \mathcal{C}$ of $k\times k$  matrices
over $\mathbb{F}$ and vectors $\vec{u} \in \mathbb{F}^{1\times k}$ and $\vec{v} \in \mathbb{F}^{k\times 1}$ such that for all words $w = w_0 w_1 \cdots w_{\ell-1} \in \mathcal{C}^*$, we have
$$f(w) = \vec{u} M_w \vec{v}^\top \quad \text{with}\quad 
M_w := M_{w_0}M_{w_1}\cdots  M_{w_\ell}.$$ 
We call $(\vec{u},M,\vec{v})$ a linear representation of $f$ and $k$ the dimension of the linear
representation of $f$.

\medskip

Corollary~\ref{co:zmensenyGraf} enables to define a subgraph~$\G$ of the graph~$G$, such that there exists a bijection between the set of all paths in~$\G$ starting and ending in~$0$ and the set of all optimal representations of Gaussian integers. 

%%%%%%%%%%%%%%%%%%%%%%%%%%%%%%%%%%%%%%%%%%%%%%%%%%
\begin{definition}\label{D:graph-GG}
We denote by~$\G$ the maximal subgraph of~$G$ satisfying the following rules:
    \begin{itemize}
        \item $q$ is a vertex of~$\G$ when $[q]$ is a vertex of~$\GG$, and
        \item $e = (q, q')$ of weight~$m$ is an edge of~$\G$ when $([q], [q'])$ is an edge of~$\GG$ with label~$m$. 
    \end{itemize}
If $P = q_0 e_1 q_1 e_2 q_2 e_3 q_3 \cdots e_{\ell} q_{\ell}$ is a path in~$\G$, then the word  $Out(e_{\ell}) \cdots Out(e_1)$  is said to be the trace of~$P$.
\end{definition}
%%%%%%%%%%%%%%%%%%%%%%%%%%%%%%%%%%%%%%%%%%%%%%%%%%

The graph~$\G$ is depicted on Figure~\ref{fig: optimalCestyP}, and all its edges are listed in Table~\ref{quaternaryT}.

%%%%%%%%%%%%%%%%%%%%%%%%%%%%%%%%%%%%%%%%%%%%%%%%%%
\begin{table}[h]
    \centering
    \setlength{\tabcolsep}{3pt}
    \renewcommand{\arraystretch}{1.2}

\begin{tabular}{|c|c|c|c|c|}
    \hline
    $q$     & $q'$ & equation fulfilled by $e$ & label $In(e)\,|\,Out(e)$ & weight $w(e)$ \\
    \hline \hline
    $1$         & $1$       & $q - \beta^2 = q' \beta^3 - 1$            & $(-1)00\,|\,00(-1)$       & $0$\\
    \hline
    $1$         & $1$       & $q - \ii\beta = q' \beta^3 - \ii$         & $0(-\ii)0\,|\,00(-\ii)$   & $0$\\
    \hline
    $1$         & $-\ii$    & $q + \beta^2 = q' \beta^3 - 1$            & $100\,|\,00(-1)$          & $0$\\
    \hline
    $1$         & $-\ii$    & $q - \beta = q' \beta^3 + \ii$            & $0(-1)0\,|\,00 \ii$       & $0$\\
    \hline
    $1$         & $1 - \ii$ & $q + \ii\beta^2 = q' \beta^3 - 1$         & $\ii 00\,|\,00(-1)$       & $0$\\
    \hline
    $1$         & $0$       & $q = q' \beta^3 + 1$                      & $000\,|\,001$             & $-1$\\
    \hline
    $\ii - 1$   & $\ii$     & $q - 1 = q' \beta^3 - \ii$                & $00(-1)\,|\,00(-\ii)$     & $0$\\
    \hline
    $\ii - 1$   & $\ii$     & $q + \ii = q' \beta^3 + 1$                & $00 \ii\,|\,001$          & $0$\\
    \hline
    $\ii - 1$   & $1$       & $q = q' \beta$                            & $0\,|\,0$                 & $0$\\
    \hline
    $0$         & $1$       & $q - \ii\beta + 1 = q' \beta^3 - \ii$     & $0(-\ii)1\,|\,00(-\ii)$   & $1$\\
    \hline
    $0$         & $1$       & $q - \ii\beta + \ii = q' \beta^3 - 1$     & $0(-\ii)\ii\,|\,00(-1)$   & $1$\\
    \hline
    $0$         & $1$       & $q - \beta^2 + 1 = q' \beta^3 -1 $        & $(-1)01\,|\,00(-1)$       & $1$\\
    \hline
    $0$         & $1$       & $q + \ii\beta^2 + \ii = q' \beta^3 - \ii$ & $\ii 0 \ii\,|\,00(-\ii)$  & $1$\\
    \hline
    $0$         & $\ii -1$  & $q - \ii\beta^2 - 1 = q' \beta^3 + 1$     & $(-\ii)0(-1)\,|\,001$     & $1$\\
    \hline
    $0$         & $0$       & $q = q' \beta$                            & $0\,|\,0$                 & $0$\\
    \hline
    $0$         & $0$       & $q + d = q' \beta^3 + d$                  & $00d\,|\,00d$             & $0$\\
    \hline
\end{tabular}
\smallskip
\caption{List of all edges of graph~$\G$, for $\beta = \ii - 1$ and $\D = \{0, \pm 1, \pm \ii\}$, up to multiplication by~$\ii^k$. An edge $e = (q, q')$ and its label $In(e)\,|\,Out(e)$ have to fulfil the equation defining edges of~$G$ from Definition~\ref{D:graph-G}, and, simultaneously, $e$~must satisfy the requirements of Definition~\ref{D:graph-GG}. Due to the symmetries described in Lemma~\ref{le:HranoveSym}, it is sufficient to describe only edges with $q = 0$ and $q' \in \{0, 1, \ii - 1\}$, and all edges with $q \in  \{1, \ii-1\}$. The parameter $d$ belongs to $\{\pm 1, \pm \ii\}$.}  
\label{quaternaryT}
\end{table}
%%%%%%%%%%%%%%%%%%%%%%%%%%%%%%%%%%%%%%%%%%%%%%%%%%

%%%%%%%%%%%%%%%%%%%%%%%%%%%%%%%%%%%%%%%%%%%%%%%%%%
\begin{figure}
    \centering
    \begin{tikzpicture}[scale=0.65]
      \newcommand{\sides}{0/$1$,90/$\imath$,180/$-1$,270/$-\imath$}
      \newcommand{\corners}{0/$1+\imath$,90/$-1+\imath$,180/$-1-\imath$,270/$1-\imath$}
      \node[shape=rectangle,draw=black,minimum size=1.3cm] (zero) at (0,0) {\(0\)};
      \foreach \r/\n in \sides
        \draw[rotate=\r] (4,0) node[shape=rectangle,draw=black,minimum size=1.3cm] {\n};
      \foreach \r/\n in \corners
        \draw[rotate=\r] (4,4) node[shape=rectangle,draw=black,minimum size=1.3cm] {\n};
      \foreach \r/\n in \sides {
        \begin{scope}[rotate=\r,style=\r]
          \draw (3,0) -- (1,0);
          \draw (-0.4,1) -- (-0.4,3);
          \draw (-0.8,1) -- (-0.8,3);
          \draw (-1,0.4) -- (-3,0.4);
          \draw (-1,0.8) -- (-3,0.8);
          \draw (-1,1) -- (-3,3);
          \draw (-3,4.4) -- (-1,4.4);
          \draw (-4,1) -- (-4,3);
          \draw (-4,-3) -- (-4,-1);
          \draw (-3.2,1) to[bend left=25] (-1,3.2);
          \draw (-1,-3.6) to[bend left=25] (-3.6,-1);
          \draw plot[smooth] coordinates {(1,4.8) (2,4.7) (2,4.5) (1,4.4)};
          \draw plot[smooth] coordinates {(-5,-0.8) (-5.2,-0.4) (-5.2,0.4) (-5,0.8)};
          \draw plot[smooth] coordinates {(-1,1) (-2.5,1.5) (-1.5,2.5) (-1,1)};
          \draw[style=zero] plot[smooth] coordinates {(-3,4.8) (0,5.4) (5.2,5.2) (4.8,1)};
        \end{scope}
      }
      \draw[style=zero] plot[smooth] coordinates {(-0.1,-0.8) (-0.56,-0.56) (-0.8,0) (-0.56,0.56) (0,0.8) (0.56,0.56) (0.8,0) (0.56,-0.56) (0.1,-0.8)};
    \end{tikzpicture}

    \caption{Scheme of graph~$\G$, for $\beta = \ii - 1$ and $\D = \{0, \pm 1, \pm \ii\}$, whose full set of edges is listed in Table~\ref{quaternaryT}. For better transparency, the full labels $In(e)\,|\,Out(e)$ of the edges~$e$ are replaced with colouring of the edges by five different colours, according to the $Out(e)$-part of the label, as follows:\\
    \centerline{$
        001 \mapsto \protect\tikz[baseline=-0.65ex] \protect\draw[0] (0,0) -- (1,0);\, ,\ 
        00\ii \mapsto \protect\tikz[baseline=-0.65ex] \protect\draw[90] (0,0) -- (1,0);\, ,\ 
        00\underline{1} \mapsto \protect\tikz[baseline=-0.65ex] \protect\draw[180] (0,0) -- (1,0);\, ,\ 
        00\underline{\ii} \mapsto \protect\tikz[baseline=-0.65ex] \protect\draw[270] (0,0) -- (1,0);\, ,\ 
        0 \mapsto \protect\tikz[baseline=-0.65ex] \protect\draw[zero] (0,0) -- (1,0);\, .
    $}
    }\label{fig: optimalCestyP} 
\end{figure}
%%%%%%%%%%%%%%%%%%%%%%%%%%%%%%%%%%%%%%%%%%%%%%%%%%

\medskip

The $3$-NAF representation of a non-zero Gaussian integer $x \in \Z[\ii]$ obtained by the transducer has a prefix $00d$, with $d \in \{\pm 1, \pm \ii\}$, and it is a concatenation of words from the set
\begin{equation}\label{E:G-edges}
\CC = \{0\} \cup \bigl\{00d : d \in \{\pm 1, \pm \ii\}\bigr\} = \{Out(e) : e \text{ is an edge of }G\} \, .
\end{equation}
Let us stress that decomposition of the $3$-NAF representation into elements of~$\CC$ is unique.\

The following lemma is a direct consequence of Corollary~\ref{co:zmensenyGraf}: 

%%%%%%%%%%%%%%%%%%%%%%%%%%%%%%%%%%%%%%%%%%%%%%%%%%
\begin{lemma}\label{le:OptimalVGrafu}
Let $x \in \Z[\ii]$. Then the number of optimal representations of~$x$ equals the number of paths~$P$ in~$\G$ such that $P$~starts and ends in vertex~$0$ and the trace of~$P$ in~$\G$ is the $3$-NAF representation of~$x$.  
\end{lemma}
%%%%%%%%%%%%%%%%%%%%%%%%%%%%%%%%%%%%%%%%%%%%%%%%%%

To establish the number of paths~$P$ in~$\G$ with a given trace, we use a standard method from graph theory.

Each element of the set $\CC = \{0, 001, 00\ii, 00\underline{1}, 00\underline{\ii}\}$ introduced in~\eqref{E:G-edges} is represented by a matrix $A_d \in \Z^{9 \times 9}$, where $d \in \D = \{0, \pm 1, \pm \ii\}$ is the last digit of~$Out(e)$. 

For every pair $q, q'$ of vertices in~$\G$, we define
\begin{equation}\label{eq:matrix-Ad}
\bigl( A_d \bigr)_{q', q} = \text{ the number of edges $e = (q, q')$ in~$\G$ with $d$ as the last letter of $Out(e)$.}
\end{equation}

When the nine vertices of~$\G$ are ordered as follows:    
\begin{equation}\label{E:matrix-Ad-order}
0, 1, \ii, -1, -\ii, 1 + \ii, -1 + \ii, -1 - \ii, 1 - \ii \, ,
\end{equation}
then the green (dashed) edges with $Out(e) = 001$ and the grey edges with $Out(e) = 0$ on Figure~\ref{fig: optimalCestyP} give
$$
A_1=\left( \begin{array}{c|cccc|cccc}
    1  &  1 & 0 & 0 & 0  &  0 & 0 & 0 & 0 \\
    \hline
    0  &  0 & 0 & 0 & 0  &  0 & 0 & 0 & 0 \\
    2  &  0 & 1 & 1 & 0  &  0 & 1 & 0 & 0 \\
    2  &  0 & 0 & 1 & 1  &  0 & 0 & 1 & 0 \\
    0  &  0 & 0 & 0 & 0  &  0 & 0 & 0 & 0 \\
    \hline
    0  &  0 & 0 & 0 & 0  &  0 & 0 & 0 & 0 \\
    1  &  0 & 0 & 1 & 0  &  0 & 0 & 0 & 0 \\
    0  &  0 & 0 & 0 & 0  &  0 & 0 & 0 & 0 \\
    0  &  0 & 0 & 0 & 0  &  0 & 0 & 0 & 0
\end{array}\right) \qquad \text{and} \qquad 
A_0=\left( \begin{array}{c|cccc|cccc}
    1  &  0 & 0 & 0 & 0  &  0 & 0 & 0 & 0 \\
    \hline
    0  &  0 & 0 & 0 & 0  &  0 & 1 & 0 & 0 \\
    0  &  0 & 0 & 0 & 0  &  0 & 0 & 1 & 0 \\
    0  &  0 & 0 & 0 & 0  &  0 & 0 & 0 & 1 \\
    0  &  0 & 0 & 0 & 0  &  1 & 0 & 0 & 0 \\
    \hline
    0  &  0 & 0 & 0 & 0  &  0 & 0 & 0 & 0 \\
    0  &  0 & 0 & 0 & 0  &  0 & 0 & 0 & 0 \\
    0  &  0 & 0 & 0 & 0  &  0 & 0 & 0 & 0 \\
    0  &  0 & 0 & 0 & 0  &  0 & 0 & 0 & 0
\end{array}\right) \, .
$$
The remaining matrices $A_d$ can be obtained, starting from matrix~$A_1$, by use of the symmetries described in Lemma~\ref{le:HranoveSym}. The task is facilitated by the symmetries fulfilled by the matrix set.\

Let us denote by $P_1, P_2, P_3 \in \Z^{4\times 4}$ the permutation matrices
$$ 
P_1=\left( \begin{array}{cccc}
0&1&0&0\\
0&0&1&0\\
0&0&0&1\\
1&0&0&0
\end{array}\right) \, , \quad P_2=\left( \begin{array}{cccc}
1&0&0&0\\
0&0&0&1\\
0&0&1&0\\
0&1&0&0
\end{array}\right) \, ,
\quad P_3=\left( \begin{array}{cccc}
0&0&0&1\\
0&0&1&0\\
0&1&0&0\\
1&0&0&0
\end{array}\right) \, ,
$$
and let us consider the following block-diagonal permutation matrices $R , C \in \Z^{9 \times 9}$:
$$
R = \left( \begin{array}{c|c|c}
   1 & & \\
   \hline
    & P_1 & \\
   \hline
    & & P_1 \\
\end{array}\right) \qquad \text{and} \qquad
C= \left( \begin{array}{c|c|c}
    1 & & \\
    \hline
    & P_2 & \\
    \hline
    &&P_3
\end{array}\right) \, .
$$
Note that
\begin{equation}\label{eq:MaticeRC}
C = C^\top \, , \qquad R^\top R = I = C^2 = R^4 \qquad \text{and} \qquad R C = C R^\top \, .
\end{equation}

Therefore, the group~$\P$ of permutation matrices generated by the two matrices $R$ and $C$ has just eight elements, namely:
\begin{equation}\label{eq:grupaPermutMatic}
    \P = \{I, R, R^2, R^3, C, C R, C R^2, C R^3\} \subset \Z^{9 \times 9} \, .
\end{equation}

%%%%%%%%%%%%%%%%%%%%%%%%%%%%%%%%%%%%%%%%%%%%%%%%%%
\begin{lemma}\label{le:vztahyMaticA_d}
Let $A_0, A_1, A_{-1}, A_{\ii}, A_{-\ii}$ be the matrices defined in~\eqref{eq:matrix-Ad}. Then 
\begin{itemize}   
    \item $A_{id} = R^\top A_d R$ for $d\in \{0, 1, i, -1, -\ii\} \,$;
    \item $A_0 = C R A_0 C \,$;
    \item $A_{\overline{d}} = R C A_d C$ for $d \in \{1, i, -1, -\ii\} \,$.   
\end{itemize}
\end{lemma}
%%%%%%%%%%%%%%%%%%%%%%%%%%%%%%%%%%%%%%%%%%%%%%%%%%
\begin{proof}
The relations follow from the ordering of vertices in~\eqref{E:matrix-Ad-order} and from Lemma~\ref{le:HranoveSym}, which implies:
\begin{itemize}
    \item $\bigl( A_{id} \bigr)_{iq', iq} = \bigl( A_d \bigr)_{q',q}$ \ for $d \in \{0, 1, i, -1, -\ii\} \,$;
    \item $\bigl (A_{0} \bigr)_{i\overline{q'}, \overline{q}} = \bigl( A_0 \bigr)_{q', q} \,$;
    \item $\bigl( A_{\overline{d}} \bigr)_{\overline{iq'}, \overline{q}} = \bigl( A_d \bigr)_{q', q}$ for $d \in \{1, i, -1, -\ii\} \,$. 
\end{itemize}   
\end{proof}
%%%%%%%%%%%%%%%%%%%%%%%%%%%%%%%%%%%%%%%%%%%%%%%%%%

%%%%%%%%%%%%%%%%%%%%%%%%%%%%%%%%%%%%%%%%%%%%%%%%%%
\begin{remark}\label{re:prohozeniMatic}
The form of the group~$\P$, together with relations between the matrices stated in Lemma~\ref{le:vztahyMaticA_d}, ensure the following properties:
\begin{itemize}
    \item for each $S_1 \in \P$, $d_1 \in \{\pm 1, \pm \ii\}$ there exist $S_2 \in \P$, ${d_2} \in \{\pm 1, \pm \ii\}$ such that $S_1 A_{d_1} = A_{d_2} S_2 \,$;
    \item for each $S_1 \in \P$ there exists $S_2 \in \P$ such that $S_1 A_{0} = A_{0} S_2 \,$.
\end{itemize}
\end{remark}
%%%%%%%%%%%%%%%%%%%%%%%%%%%%%%%%%%%%%%%%%%%%%%%%%%

The number of paths in~$\G$ via edges of a given colour with both initial and final vertex equal to zero can be expressed by means of a matrix formula. The following statement is a matrix interpretation of Lemma~\ref{le:OptimalVGrafu}.

%%%%%%%%%%%%%%%%%%%%%%%%%%%%%%%%%%%%%%%%%%%%%%%%%%
\begin{proposition}\label{prop:maticovy vzorecek}
Let $u_1 u_2 u_3 \cdots u_n \in \mathcal{C}^*$ be the $3$-NAF representation of $x \in \Z[\ii]$ and let $d_\ell$ be the last letter of $u_\ell$ for every $\ell = 1, 2, \ldots, n$.
Then the number of non-zero digits in the $3$-NAF representation is $|d_1| + |d_2| + \cdots + |d_n|$  and the number of optimal representations of~$x$ equals 
$$\vec{e} A_{d_1} A_{d_2} \cdots A_{d_n} \vec{e}^\top , \ \text{ where } \vec{e} = (1, 0, 0, \ldots, 0) \in \Z^9 \, .$$
\end{proposition}
%%%%%%%%%%%%%%%%%%%%%%%%%%%%%%%%%%%%%%%%%%%%%%%%%%

%%%%%%%%%%%%%%%%%%%%%%%%%%%%%%%%%%%%%%%%%%%%%%%%%%%%%%%%%%%%%%%%%%%%%%%%%%%%%%%%%%%%%%%%%%%%%%%%%%%%
\smallskip
\
\subsection{Computer Experiments on the Number of Optimal Representations}
\
\medskip
%%%%%%%%%%%%%%%%%%%%%%%%%%%%%%%%%%%%%%%%%%%%%%%%%%%%%%%%%%%%%%%%%%%%%%%%%%%%%%%%%%%%%%%%%%%%%%%%%%%%

Even though we mostly focus on calculating the number of optimal representations in the numeration system (1), let us now regard the number of all representations of a given number~$x$ in a numeration system with  $\beta$ being  an algebraic integer, imaginary quadratic, and let $\D$ be its minimal norm representative digit set. Of course,  a string $d_N d_{N-1} \cdots d_0$ represents the same number as the string $0 d_N d_{N-1} \cdots d_0$. In the following lemma, we shall consider only so-called  normalized representations, whose leading coefficient is non-zero. If $\D \subset \Z[\beta]$ is a  minimal norm representative digit set for~$\beta$,  then no representation of $x = 0$ contains any non-zero digit. Hence, by the previous convention, the representation of $x = 0$ is only the empty string. 

%%%%%%%%%%%%%%%%%%%%%%%%%%%%%%%%%%%%%%%%%%%%%%%%%%
\begin{lemma}\label{le:r(x,k)}
Let $\beta$ be an algebraic integer, imaginary quadratic, and let $\D$ be its minimal norm representative digit set.   
For a given $x \in \Z[\beta]$ and $k \in \N$, denote by $r(x, k)$ the number of normalized $(\beta, \D)$-representations of~$x$ using exactly~$k$ non-zero digits and having non-zero leading coefficient.  
\begin{itemize}
    \item Then $r(0, 0) = 1$ and $r(0, k+1) = 0$.
    \item If $x \neq 0$, then $r(x, 0) = 0$ and    
        \begin{equation*}
            r \left( x, k+1 \right) =\left\{ 
                \begin{array}{ll}
                    r \left( \frac{x}{\beta}, k+1 \right) & \text{if } \beta \text{ divides } x \, , \\ &\\
                    \sum\limits_{\substack{d \in \D \\ \beta | (x-d)}} r \left( \frac{x-d}{\beta}, k \right) & \text{otherwise} \, .
                \end{array}\right.
        \end{equation*}
\end{itemize}
\end{lemma}
%%%%%%%%%%%%%%%%%%%%%%%%%%%%%%%%%%%%%%%%%%%%%%%%%%
\begin{proof}
By Definition~\ref{de:minimalNorm}, if $d \in \D$ and $d \neq 0$, then $\beta$ does not divide~$d$. Hence, if $\beta$ divides~$x$, then necessarily the rightmost digit of any representation of~$x$ must be~$d_0 = 0$.

If $\beta$ does not divide~$x$ and a digit $d\in \D$ has the property $\beta | (x-d)$, then any string $s'$ containing $k$~non-zero digits and representing the number $x' = \frac{x-d}{\beta}$ can be used to compose a string~$s'd$, which represents the number~$x$ in $(\beta, \D)$.  
\end{proof}
%%%%%%%%%%%%%%%%%%%%%%%%%%%%%%%%%%%%%%%%%%%%%%%%%%

%%%%%%%%%%%%%%%%%%%%%%%%%%%%%%%%%%%%%%%%%%%%%%%%%%
\begin{example}\label{exm:stejnePocty}

Let us consider $\beta = \ii - 1$ and $\D = \{0, \pm 1, \pm \ii\}$. Since there are only two residual classes modulo~$\beta$ in~$\Z[\ii]$, the formula from Lemma~\ref{le:r(x,k)} has  for $x \neq 0$ the following form:
$$ r \left( x, k+1 \right) =
        \begin{cases}
            r \left( \frac{x}{\beta}, k+1 \right) & \ \text{if } \beta \text{ divides } x \, , \\
            \sum\limits_{d \in \{\pm 1, \pm \ii\}} r \left( \frac{x-d}{\beta}, k \right) & \text{otherwise} \, .
        \end{cases}
$$
Our digit set $\D = \{0, \pm 1, \pm \ii\}$ has the property that for every $a \in \D$ and every $\ell \in \N$ there exists $b \in \D$ such that $a \overline{\beta}^\ell = b {\beta}^\ell$. 
Hence, for every $x \in \Z[\ii]$ and $k \in \N$, the following relations hold:
\begin{equation}\label{eq:stejner(x,k)}
    r \left( x, k \right) = r \left( x \beta , k \right) = r \left( \ii\, x, k \right) = r \left( \overline{x}, k \right) \, .
\end{equation}

\noindent Let us calculate $r(x, k+1)$ for $x = 1$. If $k=1$, then 
$$r(1, 1)  =  r(0, 0) + r(2, 0) + r(1-\ii, 0) + r(1+\ii, 0) = 1.  $$
For $k\in \N$ we obtain \begin{eqnarray*}
    r(1, k+1) & = & r(0, k) + r(2, k) + r(1-\ii, k) + r(1+\ii, k) = \\
            & = & 0 + r(\ii, k) + r(-1, k) + r(-\ii, k) \ = \ 3 r(1, k)\, .
\end{eqnarray*}
By  mathematical induction we get $r(1, k+1) = 3^{k}$ for every $k \in \N$. Hence the number $x = 1$, as well as all non-zero Gaussian integers have infinitely many representations in the system $\beta = \ii - 1$ and $\D = \{0, \pm 1, \pm \ii\}$.   
\end{example}
%%%%%%%%%%%%%%%%%%%%%%%%%%%%%%%%%%%%%%%%%%%%%%%%%%

Let us return to the optimal representations. If we choose  $k$ to be the number of non-zero digits in the $3$-NAF representation of a Gaussian integer~$x$, then clearly $r(x, k)$ is the number of optimal representations of~$x$.

%%%%%%%%%%%%%%%%%%%%%%%%%%%%%%%%%%%%%%%%%%%%%%%%%%
\begin{example}\label{Ex:poctyPro2,3}
Using the formula for $r(x,k)$ from Example~\ref{exm:stejnePocty}, and by means of a computer program, we ascertain that:
\begin{enumerate}
    \item Any Gaussian integer with just 2~non-zero digits in its $3$-NAF representation has at most 3~optimal representations. It is for example the number $x = 2+\ii$,  its optimal representations are listed in Example \ref{ex:priklad2}.    
    Thanks to equation~\eqref{eq:stejner(x,k)}, each of the eight  Gaussian integers  $ \pm 1 \pm 2\ii, ¨ \pm 2 \pm \ii$  has 3~optimal representations. The  same holds also for multiples of these numbers by~$\beta^\ell$ with $\ell \in \N$. For any other Gaussian integer with just 2~non-zero digits in its $3$-NAF representation, the number of optimal representations is less than~3.
    \item Any Gaussian integer with just 3~non-zero digits in its $3$-NAF representation has at most 8~optimal representations. See e.g. $x = 5\ii-2$ with $(x)_{3-NAF} = 100\underline{1}00\underline{\ii} \, $. Analogously, the count of 8~optimal representations is reached only for $\pm x \beta^\ell, \pm \ii x \beta^\ell, \pm \overline{x} \beta^\ell, \pm \ii \overline{x} \beta^\ell \,$, with $\ell \in \N \,$. 
    \item In Table~\ref{tab:rN}, obtained by computer calculations, we provide the upper bound $\mathrm{Max}(N)$ on the number of optimal representations of Gaussian integers having just $N$~non-zero digits in their $3$-NAF representations.  
\end{enumerate}
%%%%%%%%%%%%%%%%%%%%%%%%%%%%%%%%%%%%%%%%%%%%%%%%%%
\begin{table}[h]
    \centering
    \setlength{\tabcolsep}{3pt}
    \renewcommand{\arraystretch}{1.2}
\begin{tabular}{|c||c|c|c|c|c|c|c|c|c|c|c|c|c|}
    \hline
    $N$ &$1$ &$2$ & $3$ & $4$ & $5$ & $6$ & $7$ & $8$ & $9$ & $10$ & $11$ & $12$ & $13$ \\
    \hline 
   $\mathrm{Max}(N)$ &$1$& $3$ & $8$ & $17$ & $39$ & $89$ & $201$ & $457$ & $1\,037$ & $2\,353$ & $5\,341$ & $12\,121$ & $27\,509$ \\
    \hline
\end{tabular}
\bigskip
\caption {$\mathrm{Max}(N)$ denotes the maximal number of optimal representations of Gaussian integers having just $N$~non-zero digits in their $3$-NAF representation.}\label{tab:rN}
\end{table}
%%%%%%%%%%%%%%%%%%%%%%%%%%%%%%%%%%%%%%%%%%%%%%%%%%
\end{example}
%%%%%%%%%%%%%%%%%%%%%%%%%%%%%%%%%%%%%%%%%%%%%%%%%%

Computer experiments also hint for which numbers~$x \in \Z[\ii]$ the value $\mathrm{Max}(N)$ is attained. Denote
$$ g_1 = 0 0 1 \, 0 0 \underline{\ii} \, 0 0 \underline{\ii} \, , \ \ 
g_2 = 0 0 1 \, 0 0 \underline{1} \, 0 0 \underline{\ii} \ \quad \text{and} \ \quad 
h_1 = 0 0 1 \, 0 0 \underline{\ii} \, 0 0 \underline{1} \, 0 0 \ii \, , \ \ 
h_2 = 0 0\underline{\ii} \, 0 0 \ii \, ,$$
and define four infinite eventually periodic sequences as follows:
$$\mathcal{S}_1 = g_1 \, h_1 \, h_1 \, h_1 \cdots, \quad \mathcal{S}_2 = g_2  \, h_1 \, h_1 \, h_1 \cdots, \quad \mathcal{S}_3 = g_1 \, h_2 \, h_2 \, h_2 \cdots, \quad \mathcal{S}_4 = g_2 \, h_2 \, h_2 \, h_2 \cdots .$$
For a fixed $N \in \N$, we consider without loss of generality such $x \in \Z[\ii]$ whose $3$-NAF representation has $1$ as the leading coefficient and contains just $N$ non-zero digits. Computer experiments suggest that such $x$ has $\mathrm{Max}(N)$ optimal representations if and only if the $3$-NAF representation of~$x$ is a~finite prefix of $\mathcal{S}_1, \mathcal{S}_2, \mathcal{S}_3$ or $\mathcal{S}_4$. This observation helps us to deduce the formula for values $\mathrm{Max}(N)$, see Theorem \ref{th:main}.

%%%%%%%%%%%%%%%%%%%%%%%%%%%%%%%%%%%%%%%%%%%%%%%%%%

%%%%%%%%%%%%%%%%%%%%%%%%%%%%%%%%%%%%%%%%%%%%%%%%%%

%%%%%%%%%%%%%%%%%%%%%%%%%%%%%%%%%%%%%%%%%%%%%%%%%%%%%%%%%%%%%%%%%%%%%%%%%%%%%%%%%%%%%%%%%%%%%%%%%%%%
\smallskip
\
\subsection{Gaussian Integers with Maximal Number of Optimal Representations}
\
\medskip
%%%%%%%%%%%%%%%%%%%%%%%%%%%%%%%%%%%%%%%%%%%%%%%%%%%%%%%%%%%%%%%%%%%%%%%%%%%%%%%%%%%%%%%%%%%%%%%%%%%%
%%%%%%%%%%%%%%%%%%%%%%%%%%%%%%%%%%%%%%%%%%%%%%%%%%%%%%%%%%%%%%%%%%%%%%%%%%%%%%%%%%%%%%%%%%%%%%%%%%%%

In this part, we derive an upper bound on the number of different optimal $(\beta, \D)$-representations of a Gaussian integer~$x$, depending on the number of non-zero digits in its $3$-NAF representation. We also show that this upper bound is sharp. In the case of binary signed digit representations (in base $\beta = 2$) introduced by G.~W.~Reitwiesner, the sharp upper bound is given by the sequence of Fibonacci numbers. In the numeration system we consider here, the analogous crucial role is played by the sequence $(r_N)_{N \geq 0}$ defined recursively, as follows:  
\begin{equation}\label{eq:NasFibo}
    \begin{array}{l}
        r_0 = r_1 = 1 \, , \ r_2 = 3 \, , \ r_3 = 8 \, , \ r_4 = 17 \, , \\
        r_{N+3} = r_{N+2} + 2r_{N+1} + 2r_{N} \quad \text{ for each } N \in \N \, , N \geq 2 \, .
    \end{array}
\end{equation}

\noindent Several initial elements of the sequence $(r_N)_{N \geq 1}$ coincide with the numbers listed in Table~\ref{tab:rN}. The roots of the characteristic polynomial $X^3 - X^2 - 2X - 2$ of the linear recurrent sequence $(r_N)_{N \geq 1}$ are $\lambda_1 \sim 2.26953$ and $\lambda_2 = \overline{\lambda_3}$ with $|\lambda_2| < 1$. Hence $r_N = c_1 \lambda_1^{N-1} + o(1)$, where $c_1 \sim 1.47313$.

%%%%%%%%%%%%%%%%%%%%%%%%%%%%%%%%%%%%%%%%%%%%%%%%%%

%%%%%%%%%%%%%%%%%%%%%%%%%%%%%%%%%%%%%%%%%%%%%%%%%%
\begin{theorem}\label{th:main} 
Let $N \in \N, N \geq 1$, and let us consider $(\beta, \D)$-representations of Gaussian integers in base $\beta = \ii - 1$ and alphabet $\D = \{0, \pm 1, \pm \ii\}$.
\begin{enumerate}
    \item If the $3$-NAF representation of a Gaussian integer $x \in \Z[\ii]$ contains exactly $N$~non-zero digits, then $x$~has at most $r_N$~optimal representations. 
    \item For $N\geq 4$, there are exactly 16~different Gaussian integers $x \in \Z[\ii] \setminus \beta \Z[\ii]$ whose $3$-NAF representation contains just $N$~non-zero digits and which have $r_N$ optimal representations.     
    \end{enumerate}
\end{theorem}
%%%%%%%%%%%%%%%%%%%%%%%%%%%%%%%%%%%%%%%%%%%%%%%%%%

The proof of Theorem~\ref{th:main} is split into several steps. In order to prove Theorem~\ref{th:main}, we have to determine, according to Proposition~\ref{prop:maticovy vzorecek}, the maximum 
\begin{equation}\label{eq:Max}
\mathrm{Max}(N) = \max\Bigl\{\vec{e} A_{d_1} A_{d_2} \cdots A_{d_n} \vec{e}^\top : d_j \in \{0, \pm 1, \pm \ii\} \text{ and } \sum_{j} |d_j| = N \Bigr\} \, .
\end{equation}
The number $\vec{e} A_{d_1} A_{d_2} \cdots A_{d_n} \vec{e}^\top$ can be expressed as a product of the row vector   $\vec{e} A_{d_1} A_{d_2} \cdots A_{d_k}$ and the column vector $A_{d_{k+1}} \cdots A_{d_n} \vec{e}^\top$.

We shall study which choices of the indices $d_1, d_2, \ldots, d_k$ result in the largest components of the row vector. All matrices $A_{d_j}$ have non-negative elements, and the upper left corner element is~$1$. Therefore, also the first component of the row vector $\vec{e} A_{d_1} A_{d_2} \cdots A_{d_k}$ as well as of the column vector $A_{d_{k+1}} \cdots A_{d_n} \vec{e}^\top$ is~$\geq 1$. 

%%%%%%%%%%%%%%%%%%%%%%%%%%%%%%%%%%%%%%%%%%%%%%%%%%
\begin{definition}\label{de:Usporadani}
Let $\vec{u}$ and $\vec{v}$ be row vectors from $\Z^9$.\

We denote by $\vec{u} \leq \vec{v}$ when the $j^{th}$ vector components satisfy  $(\vec{u})_j \leq (\vec{v})_j$ for each $j \in 1, \ldots, 9$.\

We say that $\vec{u}$ is  majorised  by $\vec{v}$, if there exists a permutation matrix $P \in \P$ such that $\vec{u} P \leq \vec{v}$, and we denote this fact by $\vec{u} \preceq \vec{v}$. If, moreover, the first vector components satisfy $(\vec{u}P)_1 < (\vec{v})_1$, we say that $\vec{u}$ is strictly majorised by $\vec{v}$, and denote that by $\vec{u} \prec \vec{v}$.\

When $\vec{u} \preceq \vec{v}$ and $\vec{v} \preceq \vec{u}$, we denote it by $\vec{u} \sim \vec{v}$.
\end{definition}
%%%%%%%%%%%%%%%%%%%%%%%%%%%%%%%%%%%%%%%%%%%%%%%%%%

It is easy to verify that the relation~$\preceq$ is a partial ordering on~$\Z^9$, and the relation~$\sim$ is an equivalence on~$\Z^9$. 

%%%%%%%%%%%%%%%%%%%%%%%%%%%%%%%%%%%%%%%%%%%%%%%%%%

%%%%%%%%%%%%%%%%%%%%%%%%%%%%%%%%%%%%%%%%%%%%%%%%%%
\begin{lemma}\label{le:nejlepsi}
Consider $d_1, \ldots , d_k, f_1, \ldots, f_m \in \{0, \pm 1, \pm \ii\}$ such that $\sum_{j=1}^k |d_j| = \sum_{j=1}^m |f_j|$.

If the vector $\vec{e} A_{d_1} A_{d_2} \cdots A_{d_k}$ is strictly majorised by $\vec{e} A_{f_1} A_{f_2} \cdots A_{f_m}$, then each choice of indices $d_{k+1},\ldots, d_{n} \in \{0, \pm 1, \pm \ii\}$ satisfies
$$\vec{e} A_{d_1} A_{d_2} \cdots A_{d_n} \vec{e}^\top < \mathrm{Max}(N) \, , \text{ where $N = \sum_{j=1}^n |d_j|$}.$$ 
\end{lemma}
%%%%%%%%%%%%%%%%%%%%%%%%%%%%%%%%%%%%%%%%%%%%%%%%%%
\begin{proof}
We denote by $\vec{u}$ and~$\vec{v}$, respectively, the row vectors $\vec{e} A_{d_1} A_{d_2} \cdots A_{d_k}$ and $\vec{e} A_{f_1} A_{f_2} \cdots A_{f_m}$; and we denote by~$\vec{w}$ the column vector $A_{d_{k+1}} A_{d_{k+2}} \cdots A_{d_n} \vec{e}^\top$. The first component of each of these three vectors is positive.\

As $\vec{u} \prec \vec{v}$, there exists  a permutation matrix $P \in \P$ such that $\vec{u} P \leq \vec{v}$ and $(\vec{u}P)_1 < (\vec{v})_1$. Obviously, 
$$\vec{e} A_{d_1} A_{d_2} \cdots A_{d_n} \vec{e}^\top = \vec{u} \cdot \vec{w} = (\vec{u} P) \cdot (P^\top \vec{w}) < \vec{v} \cdot (P^\top \vec{w}) \,.$$
Now, it is enough to show that $\vec{v} \cdot (P^\top \vec{w}) \leq \mathrm{Max}(N)$.\

By Remark~\ref{re:prohozeniMatic}, we have the equality $P^\top \vec{w} = A_{h_{k+1}} \cdots A_{h_n} Q \vec{e}^\top$ for a permutation matrix $Q \in \P$ and some indices $h_{k+1}, \ldots, h_{n}$. Moreover,  $|h_j| = |d_j|$ for each $j = k+1, \ldots, n$ and $Q \vec{e}^\top = \vec{e}^\top$.
Hence  
$$\vec{v} \cdot (P^\top \vec{w}) = (\vec{e} A_{f_1} A_{f_2} \cdots A_{f_m}) \cdot (A_{h_{k+1}} \cdots A_{h_n} \vec{e}^\top) \leq \mathrm{Max}(N) \,.$$
In the last inequality, we used the fact that $\sum_{j=1}^n |d_j| = \sum_{j=1}^m |f_j| + \sum_{j=k+1}^n |h_j| = N$.
\end{proof}
%%%%%%%%%%%%%%%%%%%%%%%%%%%%%%%%%%%%%%%%%%%%%%%%%%

%%%%%%%%%%%%%%%%%%%%%%%%%%%%%%%%%%%%%%%%%%%%%%%%%%
\begin{lemma}\label{lem:prvni4}
Let $N \in \{2, 3, 4\}$ and $\vec{u} = \vec{e} A_{d_1} A_{d_2} \cdots A_{d_k}$, where $|d_1| + |d_2| + \cdots + |d_k| = N$ and $d_1 \neq 0$, $d_k \neq 0$. Denote
\begin{itemize}
    \item $B_2 = A_1 A_{-1}$ and $\vec v_2 = \vec e B_2 = (3, 1, 1, 1, 0, 1, 0, 0, 0) \,$;
    \item $B_3 = A_1 A_{-1} A_{-\ii}$ and $\vec v_3 = \vec e B_3 = (8, 1, 3, 1, 3, 1, 1, 0, 0) \,$; 
    \item $B_4 = A_1 A_{-1} A_{-\ii} A_{-\ii}$ and $\vec v_4 = \vec e B_4 = (17, 1, 5, 3, 8, 1, 3, 0, 0) \,$. 
\end{itemize}
Then either $\vec v_N$ strictly majorises $\vec{u}$ or $\vec v_N \sim \vec{u}$. In the latter case, $k=N$ and there exist permutation matrices $P, S \in \P$ such that $S^\top B_N P = A_{d_1} A_{d_2} \cdots A_{d_k}$. 
\end{lemma}

\begin{proof}
Note that $A_0^2 = \vec{e} \cdot \vec{e}^\top$ -- i.e., the matrix $A_0^2$ has only one non-zero element, namely on the position $(1,1)$. Consequently, $\vec{e} A_0 = \vec{e}$ and $\vec{e} P = \vec{e}$ for every permutation matrix $P \in \P$.
Thus, in order to prove the statement, 
\begin{itemize}
    \item for $N = 2$ it suffices to inspect $\vec{u} = \vec{e} A_1 A^{\ell_1}_0 A_{d_1} \,$,
    \item for $N = 3$ it suffices to inspect $\vec{u} = \vec{e} A_1 A^{\ell_1}_0 A_{d_1} A^{\ell_2}_0 A_{d_2} \,$, 
    \item for $N = 4$ it suffices to inspect $\vec{u} = \vec{e} A_1 A^{\ell_1}_0 A_{d_1} A^{\ell_2}_0 A_{d_2} A_0^{\ell_3} A_{d_3} \,$, 
\end{itemize}
where $\ell_1, \ell_2, \ell_3\in \{0, 1\}$ and $d_1, d_2, d_3 \in \{1, -1, \ii, -\ii\}$.
This is an easy, although a bit laborious calculation.
\end{proof}
%%%%%%%%%%%%%%%%%%%%%%%%%%%%%%%%%%%%%%%%%%%%%%%%%%

%%%%%%%%%%%%%%%%%%%%%%%%%%%%%%%%%%%%%%%%%%%%%%%%%%
\begin{remark}\label{rem:dvectverice}
By inspecting the fact that $\vec{v}_4 = \vec{e} A_1 A_{-1} A_{-\ii} A_{-\ii}$ majorises other vectors, we find out that $\vec{u} = \vec{e} A_1 A_{-1} A_{-\ii} A_{1} = (17,8,3,5,1,0,3,1,0) = \vec{v}_4 C R^3$. This implies that  $\vec{u} \sim \vec{v}_4$.
Analogously, we find out that $\vec{v}_3 \sim \vec{e} A_1 A_{-\ii} A_{-\ii}$.
\end{remark}
%%%%%%%%%%%%%%%%%%%%%%%%%%%%%%%%%%%%%%%%%%%%%%%%%%

Let us now define and study a sequence of row vectors $(\vec{t}_m)_{m \geq 4}$, which fulfils the relation $\mathrm{Max}(N) = \vec{t}_N \vec{e}^\top$ for every $N \in \N$, $N \geq 4$, as we shall prove later.\  

%%%%%%%%%%%%%%%%%%%%%%%%%%%%%%%%%%%%%%%%%%%%%%%%%%
\begin{definition}\label{def:tN}
Let $\vec{t}_4 = \vec{e} A_1 A_{-1} A_{-\ii} A_{-\ii} R^3 = (17, 5, 3, 8, 1, 3, 0, 0, 1) \in \Z^9$. For each $m \in \N$, $m \geq 4$ we set
\begin{equation}\label{D:tN-via-AR}
\vec{t}_{m+1} = \vec{t}_m A_1 R^2 \,.
\end{equation}
\end{definition}
%%%%%%%%%%%%%%%%%%%%%%%%%%%%%%%%%%%%%%%%%%%%%%%%%%

%%%%%%%%%%%%%%%%%%%%%%%%%%%%%%%%%%%%%%%%%%%%%%%%%%
\begin{lemma}\label{le:tNrekurence}
Let $(r_m)_{m \geq 1}$ be the sequence defined in~\eqref{eq:NasFibo}. Then
\begin{equation}\label{eq:tN-via-rN}
t_{m} = (r_{m}, r_{m-2} + r_{m-3}, r_{m-2}, r_{m-1}, r_{m-3}, r_{m-2}, 0, 0, r_{m-3}) \, \text{ for every } m \in \N, m \geq 5 \,.
\end{equation}
\end{lemma}
%%%%%%%%%%%%%%%%%%%%%%%%%%%%%%%%%%%%%%%%%%%%%%%%%%
\begin{proof}
We proceed by induction on~$m$.

By the recurrence $r_{N+3} = r_{N+2} + 2r_{N+1} + 2r_{N}$ for each $N \in \N \, , N \geq 2$ defined in~\eqref{eq:NasFibo}, we calculate from four initial values $r_1 = 1$, $r_2 = 3$, $r_3 = 8$, $r_4 = 17$ the next value $r_5 = 39$. The formula~\eqref{D:tN-via-AR} deriving $\vec{t}_{m+1}$ from $\vec{t}_m$ provides $\vec{t}_5 = (39, 11, 8, 17, 3, 8, 0, 0, 3)$, which equals $(r_5, r_3 + r_2, r_3, r_4, r_2, r_3, 0, 0, r_2)$, thus showing the statement for $m = 5$.

Let $m \geq 5$. The induction hypothesis together with definition $\vec{t}_{m+1} = \vec{t}_m A_1 R^2$ produce
$$ \vec{t}_{m+1} = (r_{m} + 2r_{m-1} + 2r_{m-2}, r_{m-1} + r_{m-2}, r_{m-1}, r_{m}, r_{m-2}, r_{m-1}, 0, 0, r_{m-2}) \, . $$
Due to~\eqref{eq:NasFibo}, the first component of $\vec{t}_{m+1}$ equals $r_{m+1}$, thereby the statement holds also for~$m+1$.
\end{proof}
%%%%%%%%%%%%%%%%%%%%%%%%%%%%%%%%%%%%%%%%%%%%%%%%%%

%%%%%%%%%%%%%%%%%%%%%%%%%%%%%%%%%%%%%%%%%%%%%%%%%%
\begin{lemma}\label{le:lepsi}
Let $m, \ell \in \N$, $m \geq 4$, $\ell > 0$ and $d \in \{\pm 1, \pm \ii\}$. Then 
\begin{enumerate}
    \item $ \vec{t}_m A_1 \sim \vec{t}_{m+1} \,$;
    \item $ \vec{t}_m A_{-1} \prec \vec{t}_{m+1}$ and $\vec{t}_m A_{\ii} \prec \vec{t}_{m+1} \,$;
    \item $ \vec{t}_m A_{-\ii} A_d \prec \vec{t}_{m+2} \,$;
    \item $ \vec{t}_m A_0^\ell A_d \prec \vec{t}_{m+1} \,$.
\end{enumerate}
\end{lemma}
%%%%%%%%%%%%%%%%%%%%%%%%%%%%%%%%%%%%%%%%%%%%%%%%%%
\begin{proof}
We exploit Definition~\ref{def:tN} of the sequence~$(\vec{t}_m)$, Definition~\ref{de:Usporadani} of the equivalence~$\sim$, and Lemma~\ref{le:tNrekurence}. The recurrence~\eqref{eq:NasFibo} implies that $(r_m)$ is strictly increasing.
\begin{enumerate}

    \item $\vec{t}_m A_1 \sim \vec{t}_m A_1 R^2 = \vec{t}_{m+1} \,$. 

    \item Direct calculation provides the equalities 
    $$ \vec{t}_m A_{-1} = (\underbrace{r_{m} + 2r_{m-2} + 5r_{m-3}}_{< r_{m+1}}, \underbrace{r_{m-2} + 3r_{m-3}}_{< r_{m-1} + r_{m-2}}, \underbrace{r_{m-2} + r_{m-3}}_{< r_{m-1}}, r_{m}, \underbrace{r_{m-3}}_{< r_{m-2}}, \underbrace{r_{m-2} + r_{m-3}}_{< r_{m-1}}, 0, 0, \underbrace{r_{m-3}}_{< r_{m-2}}) $$
    and 
    $$ \vec{t}_m A_{\ii} R C = (\underbrace{r_{m} + 2r_{m-1} + 2r_{m-3}}_{< r_{m+1}}, \underbrace{r_{m-1} + r_{m-3}}_{< r_{m-1} + r_{m-2}}, r_{m-1}, r_{m}, \underbrace{r_{m-3}}_{< r_{m-2}}, r_{m-1}, 0, 0, \underbrace{r_{m-3}}_{< r_{m-2}}) \, .~~~~~~~~~~~~~~~~~~~~$$
    By Lemma~\ref{le:tNrekurence}, we get $\vec{t}_{m+1} = (r_{m+1}, r_{m-1} + r_{m-2}, r_{m-1}, r_{m}, r_{m-2}, r_{m-1}, 0, 0, r_{m-2})$. Thus $\vec{t}_m A_{-1}  \leq t_{m+1}$ and $ \vec{t}_m A_{\ii} R C \leq \vec{t}_{m+1}$ with strict inequality in the first component. Hence,  $\vec{t}_m A_{-1} \prec \vec{t}_{m+1}$ and $\vec{t}_m A_{\ii} \prec \vec{t}_{m+1}$.

    \item By direct calculation, we obtain 
    \begin{itemize}
        \item $\vec{t}_m A_{-\ii} A_{1} C R^2 = (r_{m} + 14r_{m-2} + 4r_{m-3}, 5r_{m-2} + r_{m-3}, 3r_{m-2} + r_{m-3}, r_{m} + 5r_{m-2} + 2r_{m-3}, r_{m-2}, 3r_{m-2} + r_{m-3}, 0, 0, r_{m-2}) \,$;
        
        \item $\vec{t}_m A_{-\ii} A_{-\ii} R^3 = (r_{m} + 14r_{m-2} + 7r_{m-3}, 5r_{m-2} + 3r_{m-3}, 3r_{m-2} + r_{m-3}, r_{m} + 5r_{m-2} + 2r_{m-3}, r_{m-2} + r_{m-3}, 3r_{m-2} + r_{m-3}, 0, 0, r_{m-2} + r_{m-3}) \,$;
        
        \item $\vec{t}_m A_{-\ii} A_{\ii} R = (3r_{m} + 7r_{m-2} + 2r_{m-3}, r_{m} + r_{m-2}, r_{m}, r_{m} + 5r_{m-2} + 2r_{m-3}, r_{m-2}, r_{m}, \\ 0, 0, r_{m-2} \,$;
        
        \item $\vec{t}_m A_{-\ii} A_{-1} C = (3r_{m} + 7r_{m-2} + 4r_{m-3}, r_{m} + r_{m-2} + r_{m-3}, r_{m}, r_{m} + 5r_{m-2} + 2r_{m-3}, r_{m-2} + r_{m-3}, r_{m}, 0, 0, r_{m-2}+r_{m-3}) \,$.
    
    \end{itemize}
    
    By comparison of the vector components, we observe the following relations:
    $$\vec{t}_m A_{-\ii} A_1 C R^2 \leq \vec{t}_m A_{-\ii} A_{-\ii} R^3 \, \text{ and } \, \vec{t}_m A_{-\ii} A_{\ii} R \leq \vec{t}_m A_{-\ii} A_{-1} C \, .$$
    If we show that $\vec{t}_m A_{-\ii} A_{-\ii} R^3 \leq \vec{t}_{m+2}$ and $\vec{t}_m A_{-\ii} A_{-1} C \leq \vec{t}_{m+2}$, the proof of item (3) shall be completed. Checking validity of the two inequalities for $m = 4, 5, 6$ is laborious, but quite straightforward, with use of Table~\ref{tab:rN} for this task. The rest of the proof is easy. Indeed, all components of the vectors are formed by linear combinations of the sequence $(r_m)$, which satisfies the recurrence relation $r_{N+3} = r_{N+2} + 2r_{N+1} + 2r_{N}$. Assume validity of the inequalities  
    $$\vec{t}_m A_{-\ii} A_{-\ii} R^3 \leq \vec{t}_{m+2} \, , \quad \vec{t}_{m+1} A_{-\ii} A_{-\ii} R^3 \leq \vec{t}_{m+3} \, , \quad \vec{t}_{m+2} A_{-\ii} A_{-\ii} R^3 \leq \vec{t}_{m+4} \, .$$
    Then, by multiplying the first and the second inequality by~$2$ and the third one by~$1$, and summing up the three multiplied inequalities, we obtain $\vec{t}_{m+3} A_{-\ii} A_{-\ii} R^3 \leq \vec{t}_{m+5} \, .$
  
    \item If $\ell \geq 2$, then $\vec{t}_m A_0^\ell = (r_{m}, \underbrace{0, \ldots, 0}_{8 \times})$, and hence $\vec{t}_m A_0^\ell A_d \sim (r_{m}, 0, 0, r_{m}, 0, 0, 0, 0, 0) \prec \vec{t}_{m+1}$ for every $d \in \{\pm 1, \pm \ii\}$.
    
    \noindent If $\ell = 1$, then $\vec{t}_m A_0 A_1 = (\underbrace{r_{m} + r_{m-2} + r_{m-3}}_{< r_{m+1}}, r_{m}, 0, \underbrace{r_{m-2} + r_{m-3}}_{< r_{m-1} + r_{m-2}}, 0, 0, 0, 0, 0) \leq \vec{t}_{m+1} R^2$.\
    Analogously, $\vec{t}_m A_0 A_d \prec \vec{t}_{m+1}$ for $d = -1, \ii, -\ii$. 

\end{enumerate} 
\end{proof}
%%%%%%%%%%%%%%%%%%%%%%%%%%%%%%%%%%%%%%%%%%%%%%%%%%

%%%%%%%%%%%%%%%%%%%%%%%%%%%%%%%%%%%%%%%%%%%%%%%%%%
\begin{proposition}\label{pro:tnJeNejlepsi} Let $(r_N)_{N\in \N}$  be the sequence defined by~\eqref{eq:NasFibo}. Then $\mathrm{Max}(N) =  r_{N}$ for every $N \in \N$.
\end{proposition}
%%%%%%%%%%%%%%%%%%%%%%%%%%%%%%%%%%%%%%%%%%%%%%%%%%
\begin{proof}

The initial cases $N = 0, 1$ are trivial. For $N = 2, 3, 4$, the statement is proved by Lemmas~\ref{le:nejlepsi} and~\ref{lem:prvni4} and from the fact that the first component of vector $\vec{v}_N$ is equal to~$r_N$ for any $N = 2, 3, 4$. For $N \geq 5$, we proceed as follows.\

In order to find $\mathrm{Max}(N)$, by Lemmas~\ref{le:nejlepsi} and~\ref{lem:prvni4}, without loss of generality we look for the maximum of the product $\vec{t}_4 \Pi \vec{e}^\top$, where $\Pi = A_{d_1} A_{d_2} \cdots A_{d_n}$ for some $d_1, d_2, \ldots, d_n \in \{0, \pm 1, \pm \ii \}$. As $\vec{t}_4 = \vec{e} A_{1} A_{-1} A_{-\ii} A_{-\ii} R^3$, we consider $N = 4 + \sum_{\ell=1}^n |d_\ell|$.\

According to Lemma~\ref{le:vztahyMaticA_d}, the matrix $\Pi$ can be expressed as a matrix product, wherein each of the multiplied matrices is $A_1$ or $R$ or~$A_0$. Denote
$$\B = \Bigl\{B_1 B_2 \cdots B_m : m \in \N, B_k \in \{A_0, A_1, R\} \, \text{ for every } k = 1, 2,\ldots, m \Bigr\} \, ,$$
$$\B_0 = \Bigl\{B_1 B_2 \cdots B_m : m \in \N, B_k \in \{A_0, R\} \, \text{ for every } k = 1, 2,\ldots, m \Bigr\} \, .$$

Let $L \in \N$ be the maximal integer such that $d_1, d_2, \ldots, d_L \neq 0$ and $\Pi =(A_1 R^2)^L B$ for some matrix $B \in \B$. In particular, $\vec{t}_4 \Pi = \vec{t}_{4+L} B$ and the number of matrices~$A_1$ occurring in~$B$ is~$N - 4 - L$. If $L = N - 4$, then $B \in \B_0$. In this case, $B \vec{e}^\top = \vec{e}^\top$ and $\vec{t}_4 \Pi \vec{e}^\top = \vec{t}_N B \vec{e}^\top = \vec{t}_N \vec{e}^\top = r_{N} \, .$\

\medskip

We have already shown that $\vec{t}_N \vec{e}^\top = r_{N} \leq \mathrm{Max}(N)$. In the rest of the proof, we show that the assumption $L < N - 4$ leads to strict inequality $\vec{t}_4 \Pi \vec{e}^\top < \mathrm{Max}(N)$. Thereby, the proof shall be completed.\
Let us discuss several cases of possible forms of the matrix~$B$. Note that our choice of $L$ forbids $B$ to be of the form $B = A_1 \tilde{B}$ for any $\tilde{B} \in \B$.
\begin{description}

    \item [Case] $B = R^j A_1 \tilde{B}$, where $j \in \{2, 3\}$ and $\tilde{B} \in \B \,$.\
    
    \noindent Then $\vec{t}_L R^j A_1 \sim \vec{t}_L A_d$, with $d \in \{-1, \ii\}$. By Lemmas~\ref{le:lepsi} and~\ref{le:nejlepsi}, we obtain the inequality $\vec{t}_4 \Pi \vec{e}^\top < \mathrm{Max}(N) \,$. 
  
   \item [Case] $B = R^1 A_1 R^j A_1 \tilde{B}$, where $j \in \{0, 1, 2, 3\}$ and $\tilde{B} \in \B \,$.\
   
   \noindent Then $\vec{t}_L R^1 A_1 R^j A_1 \sim \vec{t}_L A_{-\ii} A_d$ for some $d \in \{\pm 1, \pm \ii\}$. By Lemmas \ref{le:lepsi} and~\ref{le:nejlepsi}, it holds that $\vec{t}_4 \Pi \vec{e}^\top < \mathrm{Max}(N) \,$. 

    \item [Case] $B = R^1 A_1 R^j A_0^k A_1 \tilde{B}$, where $j \in \{0, 1, 2, 3\}$, $k \in \N, k >0$  and $\tilde{B} \in \B \,$.\

    \noindent Due to matrix elements inequality $(A_0 A_1)_{m, n} \leq (R^2A_1)_{m, n}$ for every $m, n = 1, \ldots, 9$, there is also vector elements inequality $(\vec{t}_L R^1 A_1 R^j A_0^k A_1)_n \leq (\vec{t}_L R^1 A_1 R^{j+2k} A_1)_n$ for every $n = 1, \ldots, 9$, and we apply the previous case.

    \item [Case] $B = R^1 A_1 \tilde{B}$, where $\tilde{B} \in \B_0$.\
  
    \noindent Then, necessarily, $L = N - 5$ and $\vec{t}_4 \Pi \vec{e}^\top = \vec{t}_{N-1} R^1 A_1 \tilde{B} \vec{e}^\top = \vec{t}_{N-1} R^1 A_{1} \vec{e}^\top$. Using the fact that $R^1 A_{1} \vec{e}^\top = (1, 2, 2, 0, 0, 1, 0, 0, 0)^\top$, we deduce that $\vec{t}_4 \Pi \vec{e}^\top = \vec{t}_4 R^1 A_{1} \vec{e}^\top = 36 < 39 = r_{5}$ for $N = 5$. For $N \geq 6$, the Lemma~\ref{le:tNrekurence} implies $\vec{t}_4 \Pi \vec{e}^\top = \vec{t}_{N-1} R^1 A_{1} \vec{e}^\top = r_{N-1} + 5r_{N-3} + 2r_{N-4}$. It can be checked that $r_{N-1} + 5r_{N-3} + 2r_{N-4} < r_{N}$, by recurrence~\eqref{eq:NasFibo}.  

    \item [Case] $B = A_0 \tilde{B}$, where $\tilde{B} \in \B \,$.\

    \noindent As $L < N - 4$, there exists $k \in \N, k \geq 1$ and $j \in \{0, 1, 2, 3\}$ such that $B = A_0^k R^j A_1 \hat{B}$ with $\hat{B} \in \B$. By Lemma~\ref{le:lepsi}, we get $\vec{t}_L A_0^k R^j A_1 \prec \vec{t}_{L+1}$, and thus $\vec{t}_4 \Pi \vec{e}^\top < \mathrm{Max}(N)$, as well.

\end{description}
\end{proof}
%%%%%%%%%%%%%%%%%%%%%%%%%%%%%%%%%%%%%%%%%%%%%%%%%%

%%%%%%%%%%%%%%%%%%%%%%%%%%%%%%%%%%%%%%%%%%%%%%%%%%
\begin{proof}[Proof {\rm(of Theorem \ref{th:main})}]
Lemma~\ref{lem:prvni4} gives the statement for $N = 2, 3, 4 \,$.\

Consider $N > 4$ and $x \in \Z[\ii] \setminus \beta \Z[\ii]$. The $3$-NAF representation of~$x$ is of the form $c_1 c_2\cdots c_n$ with $c_{j} \in \{00d : d \in \{\pm1, \pm \ii\}\} \cup \{ 0\}$ and  $c_n \neq 0$. The number of optimal representations of~$x$, by Proposition~\ref{pro:tnJeNejlepsi}, equals $\vec{e} A_{d_1} \cdots A_{d_n} \vec{e}^\top$. From the proof of Proposition~\ref{pro:tnJeNejlepsi}, we see that $\mathrm{Max}(N)$ is reached only if $\vec{e} A_{d_1} \cdots A_{d_n} \sim \vec{t}_N B$, where $B \in \B_0$. As $d_n \neq 0$, the matrix~$B$ is just a power of the matrix~$R$. Therefore, $\vec{e} A_{d_1} \cdots A_{d_n} \sim \vec{t}_N$.

Lemma \ref{le:lepsi} implies that our choice of $\vec{t}_4 = e A_1 A_{-1} A_{-\ii} A_{-\ii} R^3 $ allows unique prolongation of the string $001 \, 00\underline{1} \, 00\underline{\ii} \, 00\underline{\ii}$ to the $3$-NAF representation of a Gaussian integer having the maximal number of optimal representations. Due to Remark \ref{rem:dvectverice}, we have 4~equivalent possibilities how to choose the starting vector~$\vec{t}_4$. All the four possibilities correspond to the $3$-NAF strings having the leading digit equal to~$1$. Thus the four possibilities multiplied by four non-zero digits in~$\D$ give 16~distinct strings for any number $N \geq 4$ of non-zero digits in the $3$-NAF representation of $x$.
\end{proof}

%%%%%%%%%%%%%%%%%%%%%%%%%%%%%%%%%%%%%%%%%%%%%%%%%%%%%%%%%%%%%%%%%%%%%%%%%%%%%%%%%%%%%%%%%%%%%%%%%%%%
%%%%%%%%%%%%%%%%%%%%%%%%%%%%%%%%%%%%%%%%%%%%%%%%%%%%%%%%%%%%%%%%%%%%%%%%%%%%%%%%%%%%%%%%%%%%%%%%%%%%
\section{Representations of Eisenstein Integers}\label{sec:Eisenstein}
%%%%%%%%%%%%%%%%%%%%%%%%%%%%%%%%%%%%%%%%%%%%%%%%%%%%%%%%%%%%%%%%%%%%%%%%%%%%%%%%%%%%%%%%%%%%%%%%%%%%

In this chapter, we study representations of Eisenstien integers in numeration system $(\beta, \mathcal{D})$, with 
 \begin{equation}\label{eq:Eisen}
        \beta = \omega - 1  \text{  \ \ and \ \  }  \ \D = \{ 0, \pm 1, \pm \omega, \pm \omega^2 \}, \text{ where } \omega = \exp(2\pi\ii / 3)\,.
    \end{equation}
It follows from results of Heuberger and Krenn that any $x \in \mathbb{Z}[\omega]$ has a unique $2$-NAF representation, and this representation has the minimal Hamming weight among all $(\beta, \mathcal{D})$-representations of~$x$. The following theorem determines the number of optimal representations of number~$x \in \Z[\omega]$.

\begin{theorem}\label{thm:mainEisen} Let $(\beta, \mathcal{D})$ be numeration system defined in \eqref{eq:Eisen} and 
let $s_N = \lfloor \frac17 (6 \cdot 2^N +1) \rfloor$ for every  $N \in \Z$. 
\begin{itemize}
    \item If  $2$-NAF representation of $x \in \mathbb{Z}[\omega]$  has   $N$~non-zero digits, then $x$ has at most  $s_N$    optimal representations. 
    \item For $N \geq 4$, there are only 12~Eisenstein integers $x \in \Z[\omega]$, $x \notin \beta\Z[\omega]$ with $2$-NAF representation containing $N$~non-zero digits, and having  $s_{N}$ optimal representations.     
    \end{itemize}
\end{theorem}

The same method as used in the proof of Theorem~\ref{th:main} can be exploited to prove Theorem~\ref{thm:mainEisen} as well.  

Let us summarize the transducer converting a $(\beta, \D)$-representation of $x \notin \beta\Z[\omega]$ into the unique $2$-NAF representation of~$x$. The set~$Q$ of all states of the transducer coincides with the digit set, and the transducer can be visualized as a graph $G = (Q, E)$, where $q_0 = 0$ is the initial vertex.    

\begin{itemize}
\item If the transducer is in state $q \in Q$ and reads on input the digit $a \in \D$ such that $a + q = q' \beta \in \beta Q$, then it moves to the state~$q'$ and writes the digit~$0$ on output. The set $E$ of edges of the graph~$G$ contains the edge $e=(q,q')$ with label $a|0$.   
\item If $a + q \notin \beta Q$, then the transducer reads an additional digit~$b$ on input. There exists a unique $d \in \D$ such that $q + a + b \beta = q' \beta^2 + d$. The transducer moves to the state~$q'$ and writes on output the pair of digits~$0d$. In this case, $E$~contains the edge $e = (q, q')$ with label $ba|0d$.
\end{itemize}

Since the set $Q$ is invariant under multiplication by $\lambda := \exp(2\pi \imath/6) = -\omega^2$, a pair $(q,q')$ belongs to~$E$ if and only if $(\lambda q,\lambda q')$ belongs to~$E$. Moreover, we get the label of $(\lambda q, \lambda q')$ by multiplying all digits occurring in the label of $(q,q')$ by $\lambda$.

Using the symmetry of the graph, we define an equivalence relation: for $p, q \in V$ we write  
$$p \sim q  \ \ \text{if}  \ \ p = \lambda^k q \ \text{ for some $k \in \Z$}.$$
The set $Q$ of vertices is split into two equivalence classes: $[0] = \{0\}$ and $[1] = \{\pm 1, \pm \omega, \pm \omega^2 \}$.

\medskip

We can see that the transducer in this case as a lot more simple than the transducer producing $3$-NAF representations of Gaussian integers. Therefore, also the proof of Theorem~\ref{thm:mainEisen} is less technical, and need not be elaborated here in detail. We provide just a sketch of the proof, in particular, we describe only those objects whose analogues in case of Gaussian integers played important role for the proof of Theorem~\ref{th:main}. 

\medskip

%%%%%%%%%%%%%%%%%%%%%%%%%%%%%%%%%%%%%%%%%%%%%%%%%%
\begin{description}

\item[Graph $\Gamma$] \quad The graph $\Gamma$, which is a projection of the graph $G$ onto the equivalence classes $[0]$ and $[1]$ using minimal weight of the edges (analogously as in Definition~\ref{D:graph-Gamma}), in this case coincides with the graph $\GG$ (restricted  to optimal paths only), as depicted on Figure~\ref{fig:Eisenstein-graph-optimal}. It is due to the following set of edges:
        \begin{itemize}
            \item edge $e = (0, 0)$ with the label $0|0$ and  weight $w(e) = 0$;
            \item edge $e = (1, 1)$ with the label $\omega 0|0\omega$ and  weight $w(e) = 0$;
            \item edge $e= (0, 1)$ with the label $\omega 1|0 \omega$ and weight $w(e) = +1$;
            \item edge $e = (1, 0)$ with the label $ 00| 01$ and weight $w(e) = -1$.
        \end{itemize}

%%%%%%%%%%%%%%%%%%%%%%%%%%%%%%%%%%%%%%%%%%%%%%%%%%
\begin{figure}
    \centering
    \begin{tikzpicture}[scale=1.5]
        \node[shape=circle,draw=black,minimum size=13mm] (zero) at (0,0) {\([0]\)};
        \node[shape=circle,draw=black,minimum size=13mm] (one) at (2,0) {\([1]\)};
        \path[-latex]
            (zero) edge[bend right=10] node[below] {$+1$} (one)
            (one) edge[bend right=10] node[above] {$-1$} (zero)
            (zero) edge[>=latex,loop left] node[left] {$0$} ()
            (one) edge[>=latex,loop right] node[right] {$0$} ()
        ;
    \end{tikzpicture}
    \caption{Graph $\GG$ for $\beta = \omega - 1$ and $\D = \{0, \pm 1, \pm \omega, \pm \omega^2 \}$, containing optimal paths only, labeled with weights of its edges.}
    \label{fig:Eisenstein-graph-optimal}
\end{figure}
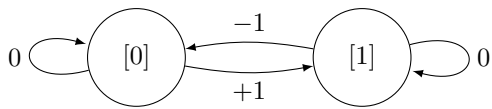
%%%%%%%%%%%%%%%%%%%%%%%%%%%%%%%%%%%%%%%%%%%%%%%%%%

\medskip

\item[Graph $\G$] \quad Analogously to Definition~\ref{D:graph-GG}, we denote by~$\G$ the subgraph of $G$ corresponding to the optimal paths in $\G$, see Fig. \ref{fig:EisensteinColor}

%%%%%%%%%%%%%%%%%%%%%%%%%%%%%%%%%%%%%%%%%%%%%%%%%%
\begin{figure}
    \centering
    \begin{tikzpicture}
      \newcommand{\sides}{0/$1$,60/$-\omega^2$,120/$\omega$,180/$-1$,240/$\omega^2$,300/$-\omega$}
      \node[regular polygon,regular polygon sides = 6,shape border rotate=90,draw=black,minimum size={2cm * 2 / sqrt(3)}] (zero) at (0,0) {\(0\)};
      \foreach \r/\n in \sides
        \draw[rotate=\r] (3,0) node[regular polygon,regular polygon sides = 6,shape border rotate=90,draw=black,minimum size={2cm * 2 / sqrt(3)}] {\n};
      \foreach \r/\n in \sides {
        \begin{scope}[rotate=\r,style=\r]
          \draw (2,0) -- (1,0);
          \draw (-1,-0.2) -- (-2,-0.2);
          \draw[rotate=60] (-1,0.4) -- (-2,0.4);
          \draw[rotate=60] (-1,-0.4) -- (-2,-0.4);
          \draw[rotate=120] (-1,0.2) -- (-2,0.2);
          \draw (-0.5,-{3/2*sqrt(3)-0.3}) -- (0.5,-{3/2*sqrt(3)-0.3});
          \draw[rotate=-60] (-0.5,-{3/2*sqrt(3)}) -- (0.5,-{3/2*sqrt(3)});
          \draw[rotate=-120] (-0.5,-{3/2*sqrt(3)+0.3}) -- (0.5,-{3/2*sqrt(3)+0.3});
          \draw plot[smooth] coordinates {(-2.5,{3/2*sqrt(3)}) (-4.4,2) (-4.4,-2) (-2.5,{-3/2*sqrt(3)})};
          \draw[rotate=60] plot[smooth] coordinates {(-4,0.4) (-4.3,0.2) (-4.3,-0.2) (-4,-0.4)};
          \draw[rotate=-60] plot[smooth] coordinates {(0,{2/sqrt(3)}) (-0.3,1.9) (0.3,1.9) (0,{2/sqrt(3)})};
        \end{scope}
      }
      \draw[style=zero] plot[smooth] coordinates {(-0.1,-0.8) (-0.56,-0.56) (-0.8,0) (-0.56,0.56) (0,0.8) (0.56,0.56) (0.8,0) (0.56,-0.56) (0.1,-0.8)};
    \end{tikzpicture}
    \caption{Scheme of graph~$\G$ for $\beta = \omega - 1$ and $\D = \{0, \pm 1, \pm \omega, \pm \omega^2 \}$. For better transparency, the full labels $In(e)\,|\,Out(e)$ of the edges~$e$ are replaced with colouring of the edges by seven different colours, according to the $Out(e)$-part of the label, as follows:\\
    \centerline{$
        01 \mapsto \protect\tikz[baseline=-0.65ex] \protect\draw[0] (0,0) -- (1,0);,
        0\underline{\omega}^2 \mapsto \protect\tikz[baseline=-0.65ex] \protect\draw[60] (0,0) -- (1,0);,
        0\omega \mapsto \protect\tikz[baseline=-0.65ex] \protect\draw[120] (0,0) -- (1,0);,
        0\underline{1} \mapsto \protect\tikz[baseline=-0.65ex] \protect\draw[180] (0,0) -- (1,0);,
    $}\\\centerline{$
        0\omega^2 \mapsto \protect\tikz[baseline=-0.65ex] \protect\draw[240] (0,0) -- (1,0);,
        0\underline{\omega} \mapsto \protect\tikz[baseline=-0.65ex] \protect\draw[300] (0,0) -- (1,0);,
        0 \mapsto \protect\tikz[baseline=-0.65ex] \protect\draw[zero] (0,0) -- (1,0);.
    $}
    }\label{fig:EisensteinColor}
\end{figure}
%%%%%%%%%%%%%%%%%%%%%%%%%%%%%%%%%%%%%%%%%%%%%%%%%%

\medskip

\item[Matrices $A_d$] \quad When the seven vertices of~$\G$ are ordered as follows:
    \begin{equation}\label{E:matrix-Ad-order-Eisenstein}
        0, +1, -\omega^2, +\omega, -1, +\omega^2, -\omega \ = \ 0, \lambda^0, \lambda^1, \lambda^2, \lambda^3, \lambda^4, \lambda^5, \text{ where }  \lambda = \exp{\frac{2\pi \imath }{6}},
    \end{equation}
and when defining the matrices ${A}_d$ analogously as in (\ref{eq:matrix-Ad}) by
    \begin{equation}\label{eq:matrix-Ad-Eisenstein}
        \bigl({A}_d \bigr)_{q', q} = \text{ the number of edges $e=(q, q')$ in~$\G$ with $d$ as the last letter of $Out(e)$,}
    \end{equation}
we get
    $${A}_0 = \left( \begin{array}{c|cccccc}
        1  &  0 & 0 & 0 & 0 & 0 & 0 \\
        \hline
        0  &  0 & 0 & 0 & 0 & 0 & 0 \\
        0  &  0 & 0 & 0 & 0 & 0 & 0 \\
        0  &  0 & 0 & 0 & 0 & 0 & 0 \\
        0  &  0 & 0 & 0 & 0 & 0 & 0 \\
        0  &  0 & 0 & 0 & 0 & 0 & 0 \\
        0  &  0 & 0 & 0 & 0 & 0 & 0
    \end{array}\right) \, 
    \qquad \text{and} \qquad 
    {A}_1 = \left( \begin{array}{c|cccccc}
        1  &  1 & 0 & 0 & 0 & 0 & 0 \\
        \hline
        0  &  0 & 0 & 0 & 0 & 0 & 0 \\
        0  &  0 & 0 & 0 & 0 & 0 & 0 \\
        0  &  0 & 0 & 0 & 0 & 0 & 0 \\
        1  &  0 & 0 & 1 & 0 & 0 & 0 \\
        2  &  0 & 0 & 1 & 1 & 1 & 0 \\
        1  &  0 & 0 & 0 & 0 & 1 & 0
    \end{array}\right) \, ,$$
and, for any other $d \in \D$, the matrix ${A}_d$ can be obtained with formula $R^\top {A}_d R = {A}_{\lambda d}$ by means of permutation matrix
    $$R = \left( \begin{array}{c|cccccc}
        1  &  0 & 0 & 0 & 0 & 0 & 0 \\
        \hline
        0  &  0 & 1 & 0 & 0 & 0 & 0 \\
        0  &  0 & 0 & 1 & 0 & 0 & 0 \\
        0  &  0 & 0 & 0 & 1 & 0 & 0 \\
        0  &  0 & 0 & 0 & 0 & 1 & 0 \\
        0  &  0 & 0 & 0 & 0 & 0 & 1 \\
        0  &  1 & 0 & 0 & 0 & 0 & 0
    \end{array}\right) \, , \qquad R^6 = I .$$

\medskip 

\item[Sequence $(\vec{t}_N)$ of vectors in $\mathbb{Z}^7$] \quad 
Computer experiments suggested which Eisenstein integers with $N$ non-zero digits in their $2$-NAF  have the maximal number of optimal representations. To specify these $2$-NAF representations, we adopt again the convention that $\underline{d}$ stands for the digit $-d$, and denote 
$$ g = 0 1 0 \omega 0 \underline{\omega}^2 \qquad \text{and} \qquad h = 0 \underline{1} 0 \underline{\omega} 0 \omega^2 \,. $$
We observed that if the $2$-NAF representation of an Eisenstein integer~$x$ is a prefix of the infinite purely periodic word $\mathcal{S} = g \, h \, g \, h \, g \, h \, g \, h \cdots$ or of the infinite periodic word $g \, \mathcal{S} = g \, g \, h \, g \, h \, g \, h \, g \, h \cdots$, then $x$~has the maximal number of optimal representations. 

Definition of the vectors $\vec{t}_N$ reflects the form of the infinite word~$\mathcal{S}$. Using the period~$g \, h$ of the word~$\mathcal{S}$, and due to relation
$$A_{+1} A_{+\omega} A_{-\omega^2} A_{-1} A_{-\omega} A_{+\omega^2} = (A_1 R) (A_{-1} R) (A_{-1} R) (A_1 R) (A_{-1} R) (A_{-1} R) \, ,$$
\noindent we define the sequence $(\vec{t}_N)$ of row vectors in $\Z^7$ recursively as follows: $\vec{t}_0 = \vec{e} = (1,0,\ldots, 0)$, and for every $N \in \N$ we put 
$$\vec{t}_{N+1} = 
    \left\{ \begin{array}{ll} 
    \vec{t}_N A_1 R & \text{ if } \ N \equiv 0 \mod {3} \, ;\\
    \vec{t}_N A_{-1} R & \text{ otherwise.}
    \end{array}\right.$$
Elements of the vector $\vec{t}_N$ can be expressed by using the sequence $s_N = \lfloor \frac17 (6 \cdot 2^N +1) \rfloor$, $N \in \mathbb{Z}$. Note that $s_N = 0$ for every $N < 0$. By induction on~$N$, one can prove for every $N \in \N$ that
$$\vec{t}_{N} = 
    \left\{ \begin{array}{ll} 
    (s_N, 2s_{N-3}, s_{N-3}, 2s_{N-3}, 0, s_{N-1}, 0) & \text{ if } \ N \equiv 0 \mod {3} \, ;\\
    (s_N, 0, s_{N-1}, 0, s_{N-2}, s_{N-2}, s_{N-2})& \text{ if } \ N \equiv 1 \mod {3} \, ;\\
    (s_N, s_{N-2}, s_{N-2}, s_{N-2}, 0, s_{N-1}, 0)& \text{ if } \ N \equiv 2 \mod {3} \, . 
    \end{array}\right.$$

\noindent The described explicit form of $\vec{t}_N$ enables to show for $d \in \mathcal{D}$ and $N \in \mathbb{N}$, $N>0$, that
\begin{equation}\label{eq:NerovnostEisen}
    \begin{array}{llllll}

            \vec{t}_N A_d & \prec & \vec{t}_N A_1 & \sim \ \vec{t}_{N+1} &       \text{\ \ if \ $N \equiv 0\!\!\! \mod {3} \, $} \quad  \text{and} &  \text{ $d \neq 1$,}\\

            \vec{t}_N A_d & \prec & \vec{t}_N A_{-1} & \sim \ \vec{t}_{N+1} &         \text{\ \ if \ $N \equiv 1 \!\!\! \mod {3} \,$} \quad \text{and} &  \text{ $d \neq -1 $,}\\

            \vec{t}_N A_d & \prec & \vec{t}_N A_{-1} & \sim \ \vec{t}_{N+1} &         \text{\ \ if \ $N \equiv 2 \!\!\! \mod {3} \, $} \quad \text{and}&\text{ $d \neq \pm 1 $,}\\
        
            \vec{t}_N A_1 A_{d} & \prec & \vec{t}_N A_{-1} A_1 & \sim \ \vec{t}_{N+2} & \text{\ \ if \ $N \equiv 2 \!\!\! \mod {3} \, $},  \quad & \\

    \end{array}
\end{equation}
\noindent where the relation $\prec$ is introduced analogously to Definition \ref{de:Usporadani} with permutation set $\mathcal{P} = \{R^k: k \in \N\}$. 

\end{description}
%%%%%%%%%%%%%%%%%%%%%%%%%%%%%%%%%%%%%%%%%%%%%%%%%%

\medskip

The $2$-NAF representation of $x \in \Z[\omega]$  can be viewed as a string $u_1 u_2 u_3 \cdots u_n$  over the alphabet 
\begin{equation}\label{E:G-edges-Eisenstein}
    \CC = \{0\} \cup \bigl\{0d : d \in \{\pm 1, \pm \omega, \pm \omega^2 \}\bigr\}\, .
\end{equation}
If  $d_\ell$ denotes  the last letters of $u_\ell$ for every $\ell = 1, 2, \ldots, n$, then the number of non-zero digits in the $2$-NAF representation is $|d_1| + |d_2| + \cdots + |d_n|$  and the number of optimal representations of~$x$ equals 
$$\vec{e} {A}_{d_1} {A}_{d_2} \cdots {A}_{d_n} \vec{e}^\top , \ \text{ where } \vec{e} = (1, 0, \ldots, 0) \in \Z^7 \, .$$
Fix $N = |d_1| + |d_2| + \cdots + |d_n|$. Relation \eqref{eq:NerovnostEisen} implies $\vec{e} {A}_{d_1} {A}_{d_2} \cdots {A}_{d_n} \vec{e}^\top \leq \vec{t}_N \vec{e}^\top = s_N$, which proves the first part of Theorem \ref{thm:mainEisen}. The second part of the theorem is a consequence of the fact that $x \in \mathbb{Z}[\omega]$ and $\lambda^k x$ have, due to symmetries of $\mathbb{Z}[\omega]$, the same number of optimal representations. Thus, if we multiply digits of a prefix of the infinite sequence $\mathcal{S}$ or $u\mathcal{S}$ by $\lambda^k$ for $k=0,1,\ldots, 5$ we get  $2$-NAF representation of an Eisenstein integer with maximal number of optimal representations.


\begin{thebibliography}{20}
%%%%%%%%%%%%%%%%%%%%%%%%%%%%%%%%%%%%%%%%%%%%%%%%%%%%%%%%%%%%%%%%%%%%%%%%%%%%%%%%%%%%%%%%%%%%%%%%%%%%

\bibitem{BeRe2010} J.~Berstel, Ch.~Reutenauer: Noncommutative Rational Series, Cambridge University Press, 2010.

\bibitem{Avi} A.~Avizienis: Signed-digit number representations for fast parallel arithmetic, {\em IRE Trans. Electron. Comput.} {\bf 10} (1961) 389--400.

\bibitem{ChR1978} C.~Y.~Chow, J.~E.~Robertson: Logical design of a redundant binary adder, {\em Proc. 4th IEEE Symposium on Computer Arithmetic} (1978) 109--115.

\bibitem{FrPaPeSv} Ch.~Frougny, M.~Pavelka, E.~Pelantov\'a, M.~Svobodov\'a: On-line algorithms for multiplication and division in real and complex numeration systems, {\em Discr. Math. Theor. Comput. Sci.} {\bf 21}(3) (2019) \#14.

\bibitem{FrPeSv} Ch.~Frougny, E.~Pelantov\'a, M.~Svobodov\'a: Minimal digit sets for parallel addition in non-standard numeration systems, {\em Jour. Integ. Seq.} {\bf 16}(2) (2013) A13.2.17.

\bibitem{GraHeu2006} P.~Grabner, C.~Heuberger: On the number of optimal base 2 representation of integers, {\em Designs, Codes and Cryptography} {\bf 40} (2006) 25–-39.

\bibitem{HK2013} C.~Heuberger, D.~Krenn: Existence and optimality of w-non-adjacent forms with an algebraic integer base, {\em Acta Math. Hung.} {\bf 140} (2013) 90–-104.

\bibitem{HK2013b} C.~Heuberger, D.~Krenn: Analysis of width-w non-adjacent forms to imaginary quadratic bases, {\em Jour. of Number Theory} {\bf 133}(5) (2013) 1752--1808.

\bibitem{L2018} J.~Legerský: Minimal non-integer alphabets allowing parallel addition, {\em Acta Polytechnica} {\bf 58}(5) (2018) 285--291.

\bibitem{LS2019} J.~Legerský, M.~Svobodová: Construction of algorithms for parallel addition in expanding bases via Extending Window Method, {\em Theor. Comput. Sci.} {\bf 795} (2019) 547--569.

\bibitem{Parhami} B.~Parhami: On the Implementation of Arithmetic Support Functions for Generalized Signed-Digit Number Systems, {\em IEEE Trans. Computers} {\bf 42}(3) (1993) 379--384.

\bibitem{P1965} W.~Penney: A "binary" system for complex numbers, {\em J.A.C.M.} {\bf 12} (1965) 247--248.

\bibitem{R1960} G.~W.~Reitwiesner: Binary Arithmetic, {\em Advances in Computers} {\bf 1} (1960) 231--308.

\bibitem{TrEr} K.~S.~Trivedi, M.~D.~Ercegovac: On-line algorithms for division and multiplication, {\em IEEE Transactions on Computers} {\bf C-26} (1977) 681--687.

%\bibitem{TV2015} J.~Tůma, J.~Vábek: On the number of binary signed digit representations of a given weight, {\em CMUC} {\bf 56} (2015) 287--306.

\bibitem{WZWD2010} T.~Wu, M.~Zhang, H.~Du, R.~Wang: On optimal binary signed digit representation of integers, {\em Appl. Math. Jour. Chinese Univ.} {\bf B-25}(3) (2010) 331--340.

\end{thebibliography}
\end{document}